\documentclass[reqno]{amsart}
\usepackage{amsmath}
\usepackage{amsthm, amscd, amssymb, amsfonts, amsbsy}
\usepackage{enumerate}
\usepackage{verbatim}
\usepackage{esint}

\date{\today}

\numberwithin{equation}{section}

\theoremstyle{plain}
\newtheorem{theorem}{Theorem}[section]
\newtheorem{lemma}[theorem]{Lemma}
\newtheorem{corollary}[theorem]{Corollary}
\theoremstyle{definition}
\newtheorem{assumption}[theorem]{Assumption}
\newtheorem{definition}[theorem]{Definition}

\theoremstyle{remark}
\newtheorem{remark}[theorem]{Remark}
\newtheorem{nota}{Notation}

\newcommand{\bR}{\mathbb{R}}
\newcommand{\bS}{\mathbb{S}}

\newcommand{\cC}{\mathcal{C}}

\newcommand{\cD}{\mathcal{D}}

\newcommand{\cL}{\mathcal{L}}
\newcommand{\cM}{\mathcal{M}}

\newcommand{\p}{\partial}

\newcommand{\tOmega}{\widetilde{\Omega}}

\newcommand{\set}[1]{\left\lbrace #1 \right\rbrace}

\newcommand{\rY}{\mathring{Y}}

\DeclareMathOperator*{\divg}{div}
\DeclareMathOperator*{\Lip}{Lip}
\DeclareMathOperator*{\VMO}{VMO} 
\DeclareMathOperator*{\BMO}{BMO} 

\begin{document}
\title[Fundamental Solution]{Fundamental solutions for stationary Stokes systems with measurable coefficients}

\author[J. Choi]{Jongkeun Choi}
\address[J. Choi]{Department of Mathematics, Korea University, 145 Anam-ro, Seongbuk-gu, Seoul, 02841, Republic of Korea}
\email{jongkeun\_choi@korea.ac.kr}

\author[M. Yang]{Minsuk Yang}
\address[M. Yang]{School of Mathematics, Korea Institute for Advanced Study, 85 Hoegi-ro Dongdaemun-gu, Seoul 130-722, Republic of Korea}
\email{yangm@kias.re.kr}


\begin{abstract}
We establish the existence and the pointwise bound of the fundamental solution for the stationary Stokes system with measurable coefficients in the whole space $\mathbb{R}^d$, $d \ge 3$,
under the assumption that weak solutions of the system are locally H\"older continuous. 
We also discuss the existence and the pointwise bound of the Green function for the Stokes system with measurable coefficients on  $\Omega$, where $\Omega$ is an unbounded domain such that the divergence equation is solvable.
Such a domain includes, for example, half space and an exterior domain.
\end{abstract}

\maketitle

\section{Introduction}
\label{S1}
In this paper, we study the stationary Stokes system 
\begin{equation} 		\label{170322@eq1}
\left\{
\begin{aligned}
&\cL u+\nabla p=f \\
&\divg u=g
\end{aligned}
\right.
\end{equation}
in $\bR^d$, $d\ge 3$ and half space where $\cL$ is an elliptic operator 
\[\cL u=-D_\alpha (A^{\alpha\beta}D_\beta u)\]
acting on vector fields $u=(u^1,\ldots,u^d)^{\operatorname{tr}}$.
Throughout the paper we use Einstein's summation convention over repeated indices.
The coefficients $A^{\alpha\beta}=A^{\alpha\beta}(x)$ are $d\times d$ matrix valued functions whose entries $A^{\alpha\beta}_{ij}(x)$ are bounded and satisfy the strong ellipticity condition, 
i.e., there exists a constant $\lambda\in (0,1)$ such that for any $x\in \bR^d$ and 
$\xi=(\xi^i_\alpha),\, \eta=(\eta^i_\alpha)\in \bR^{d\times d}$, we have 
\begin{equation}		\label{161123@EQ1}
\sum_{\alpha,\beta,i,j=1}^dA^{\alpha\beta}_{ij}(x)\xi^j_\beta \xi^i_\alpha\ge \lambda|\xi|^2, \quad \sum_{\alpha,\beta,i,j=1}^d\big|A^{\alpha\beta}_{ij}(x)\xi^j_\beta \eta^i_\alpha\big|\le \lambda^{-1}|\xi||\eta|.
\end{equation}

Let $\cM : \Omega \to \tOmega$ be a smooth diffeomorphism whose Jacobian equals one to preserve the incompressibility of the flow.
If we set $v(y) = v(\cM(x)) = u(x)$ and $q(y) = q(\cM(x)) = p(x)$ for all $y = \cM(x) \in \tOmega$, then we have for $i=1,2,\dots,d$
\[\frac{\p v}{\p y^i} = \frac{\p u}{\p x^j} \frac{\p (\cM^{-1})^j}{\p y^i}, \qquad \frac{\p q}{\p y^i} = \frac{\p p}{\p x^j} \frac{\p (\cM^{-1})^j}{\p y^i}.\]
We may regard the directional deriavatives as a gradient operator $\nabla_y = \p \cM^{-1} \nabla_x$.
Using this operator we can write $\divg_y v = (\p \cM^{-1} \nabla_x) \cdot u$ and so $\divg v=g$ is equivalent to 
\[\divg u = (\p \cM) g.\]
Similarly, we can rewrite $-\Delta_y v  + \nabla_y q = f$ as 
\[ \nabla_x \cdot (\p \cM^{-1} \nabla_x u) + \nabla_x p = (\p \cM) f.\]
This situation often occurs when one consider the limiting case of the Stokes system in time varying domains. 
These variable coefficient systems are used also for describing inhomogeneous fluids with density dependent viscosity (see, for instance, \cite{MR2663713,MR0425391}).
Giaquinta--Modica \cite{MR641818} gave various regularity results for nonlinear systems of the type of the stationary Navier--Stokes system.
$L_p$-estimates of these operators were established recently in \cite{arXiv:1503.07290v3, arXiv:1604.02690v2,arXiv:1702.07045v1}.
This motivates our study of the Stoke system with variable coefficients.

For the classical Stokes system 
\[-\Delta u+\nabla p=f, \qquad \divg u=g\]
there are a huge number of literatures regarding the Green function, which plays a significant role in the study of mathematical fluid dynamics.
One of the most popular references is a monograph \cite{MR2808162} written by Galdi.
We refer the reader for additional discussions of the fundamental solution to \cite{MR1683625,MR1282728} and references therein.
For the study of the Green function subject to Dirichlet boundary conditions on bounded domains in $\bR^2$ or $\bR^3$, we refer to \cite{MR725151,MR734895,MR2465713,MR2718661,MR2763343} and references therein. 
For mixed boundary value problems in $\bR^3$, Maz'ya--Rossmann \cite{MR2182091} obtained the pointwise estimate of Green functions.
For the two dimensional case, Ott--Kim--Brown \cite{MR3320459} obtained corresponding results.

Our aim is to construct the fundamental solution $(V(x,y), \Pi(x,y))$ and to establish the pointwise bound of $V(x,y)$
\begin{equation}		\label{161130@eq1}
|V(x,y)|\le C_0|x-y|^{2-d}, \quad \forall x,\,y\in \bR^d, \quad 0<|x-y|\le R_0
\end{equation}
under the assumption that weak solutions $(u,p)$ of either 
\begin{equation}		\label{170325@eq1}
\cL u+\nabla p =0, \qquad \divg u=0
\end{equation}
or 
\begin{equation}		\label{170325@eq2}
\cL^* u+\nabla p =0, \qquad \divg u=0
\end{equation}
are locally H\"older continuous, where $\cL^*$ denotes the adjoint operator 
\[\cL^* u=-D_\alpha(A^{\beta\alpha}(x)^{\operatorname{tr}}D_\beta u).\]
We shall show that the local H\"older continuity assumption is satisfied even in the following general cases.
\begin{enumerate}[i)]
\item
The coefficients $A^{\alpha\beta}$ are merely measurable functions of only one fixed direction.
\item
The coefficients $A^{\alpha\beta}$ are partially $\BMO$ (measurable in one direction and having small $\BMO$ semi norms in the other variables). 
\end{enumerate}
The first case is actually a special case of the second one.
However, the pointwise estimate \eqref{161130@eq1} holds for all $R_0\in (0, \infty)$ for the case i), whereas \eqref{161130@eq1} holds for some $R_0$ for the case ii); see Section \ref{S2} for more explicit statements. 
We are also interested in the existence and the global pointwise bound of the Green function for the Stokes system \eqref{170322@eq1} in an unbounded domain $\Omega\subset \bR^d$, $d\ge 3$.
We prove that if the problem 
\[
\left\{
\begin{aligned}
\divg u=g \quad \text{in }\, \Omega,\\
u=0 \quad \text{on }\, \partial \Omega,\\
\|Du\|_{L_q(\Omega)}\lesssim \|g\|_{L_q(\Omega)}
\end{aligned}
\right.
\]
is solvable and if  weak solutions of the system \eqref{170325@eq1} or \eqref{170325@eq2} are locally H\"older continuous, then the Green function exists and satisfies a natural growth estimate near the pole; see Theorems \ref{MRB} and \ref{MRH}.
Morever, we obtain the global pointwise bound for  the Green function under an additional assumption that weak solutions of Dirichlet problem are locally bounded up to the boundary; see Theorems \ref{MRE} and \ref{MRI}.

Unlike the classical Stokes system with the Laplace operator,  we are not able to find any literature explicitly dealing with the existence and the pointwise estimate of the fundamental solution for the Stokes system with nonsmooth coefficients.
In a recent article \cite{arXiv:1503.07290v3}, the existence of the Green function for the general Stokes system with $\VMO$ (vanishing mean oscillation) coefficients in a bounded Lipschitz domain has been studied.
We note  that in this paper, interior and boundary estimates for the pressure $\Pi(x,y)$ of the Green function are established with precise information on the dependence of the estimates, whereas in \cite{arXiv:1503.07290v3}  $L^q$-integrability on a domain for the pressure of the Green function is considered.

Green functions for the linear systems have been studied by many authors.
In particular, Hofmann--Kim \cite{MR2341783} proved the existence and various estimates of the Green function for the elliptic system with irregular coefficients on any open domain.
Kang--Kim \cite{MR2718661} established the global pointwise estimate of the Green function for the system.
We also refer the reader to \cite{MR3261109, MR3105752} for the study of Green functions for elliptic systems with irregular coefficients subject to Neumann or Robin boundary condition.
In this paper, we mainly follow the arguments by Hofmann--Kim \cite{MR2341783} and Kang--Kim \cite{MR2718661}, but the technical details are different from those papers because the presence of the pressure term $p$ makes the argument more involved.
In order to estimate $V(x,y)$ and $\Pi(x,y)$, we utilize the solvability of the divergence equation in the domain.

The organization of this paper is as follows.
In Section \ref{S2}, we set up our notations and state our main results.
In Section \ref{S4}, we gather some auxiliary lemmas.
From Section \ref{S5} to Section  \ref{S10}, we give each proof of our main theorems, Theorem \ref{MRA}, Theorem \ref{MRB}, Theorem \ref{MRD}, Theorem \ref{MRE}, Theorem \ref{MRF}
, and Theorem \ref{MRG}.
Section \ref{S11} is devoted to the study of the Green function on an unbounded domain such as an exterior domain.

Throughout the paper we shall use the following notation.

\begin{nota}
We denote $A \lesssim B$ if there exists a generic positive constant $C$ such that $|A| \le C|B|$.
We add subscript letters like $A \lesssim_{a,b} B$ to indicate the dependence of the implied constant $C$ on the parameters $a$ and $b$.
\end{nota}

\section{Main results}
\label{S2}
Before stating our main results, we set up some notations and definitions.
We use $x=(x_1,x')=(x_1,\ldots,x_d)$ to denote a point in $\bR^d$.
We fix half space to be 
\[\bR^d_+=\{x=(x_1,x')\in \bR^d:x_1>0,\, x'\in \bR^{d-1}\}.\]
We denote by $B_r(x)$ usual Euclidean balls of radius $r>0$ centered at $x\in \bR^d$ and 
by $B_r^+(x)$ half balls 
\[B_r^+(x)=\{y\in B_r(x):y_1>x_1\}.\]
Balls in $\bR^{d-1}$ are denoted by $B_r'(x')=\{y'\in \bR^{d-1}:|x'-y'|<r\}$.
We use the following abbreviations $B_r=B_r(0)$ and $B_r^+=B_r^+(0)$, where $0\in \bR^d$, and $B_r'=B_r'(0)$, where $0\in \bR^{d-1}$.
We use the standard notation for spheres $\bS^{d-1} = \set{e \in \bR^d : |e|=1}$.
We define $d_x=\operatorname{dist}(x,\partial \Omega)$ for $x\in \Omega$ and $d_x=\infty$ if $\Omega=\bR^d$.

\begin{definition}[$Y^1_q(\Omega)$ spaces] 
\label{D21}
Let $d\ge 3$ and $\Omega$ be an open set in $\bR^d$.
The space $Y^1_q(\Omega)$ is defined for $q\in [1,d)$ to be the family of all weakly differential functions $u\in L_{dq/(d-q)}(\Omega)$ whose weak derivatives are functions in $L_q(\Omega)$.
The space $Y^1_q(\Omega)$ is endowed with the norm 
\[\|u\|_{Y^1_q(\Omega)}=\|u\|_{L_{dq/(d-q)}(\Omega)}+\|Du\|_{L_q(\Omega)}.\]
We let $\mathring{W}^1_q(\Omega)$ and $\rY^1_q(\Omega)$ be the closure of $C_c^\infty(\Omega)$ in $W^1_q(\Omega)$ and $Y^1_q(\Omega)$, respectively. 
Here $W^1_q(\Omega)$ denotes the usual Sobolev space.
\end{definition}

\begin{remark}
We note that $Y^1_q(\bR^d)=\rY^1_q(\bR^d)$ (see  \cite[p. 46]{MR1461542}).
The Sobolev inequality implies that for all $u\in \rY^1_q(\Omega)$
\[\|u\|_{L_{dq/(d-q)}(\Omega)} \lesssim_{d,q} \|Du\|_{L_q(\Omega)}.\]
Therefore, $\rY^1_2(\Omega)$ can be understood as a Hilbert space with the inner product
\[\langle u,v\rangle=\int_\Omega D_\alpha u\cdot D_\alpha v\,dx.\]
\end{remark}

\begin{nota}
We denote an average of a function $u$ on $\Omega$ by 
\[(u)_\Omega=\fint_\Omega u\,dx=\frac{1}{|\Omega|}\int_\Omega u\,dx.\]
\end{nota}

\begin{definition}[Weak solutions]
\label{D22}
Let
\[
f\in L_{2d/(d+2)}(\Omega)^d, \quad f_\alpha\in L_2(\Omega)^d, \quad g\in L_2(\Omega).
\]
We say that $(u, p)\in \rY^1_2(\Omega)^d \times L_2(\Omega)$ is a weak solution to 
\[\left\{
\begin{aligned}
\cL u+\nabla p =f+D_\alpha f_\alpha &\quad \text{in }\, \Omega,\\
\divg u=g &\quad \text{in }\, \Omega,
\end{aligned}
\right.\]
in an unbounded domain $\Omega=\bR^d$ or $\Omega=\bR_+^d$ if $(u, p)$ satisfies the system in the sense of distributions in $\Omega$.
In particular, for any $\phi\in \rY^1_2(\Omega)^d$ 
\[\int_\Omega A^{\alpha\beta}D_\beta u\cdot D_\alpha \phi\,dx - \int_\Omega p \divg \phi\,dx=\int_\Omega f\cdot \phi\,dx-\int_\Omega f_\alpha\cdot D_\alpha \phi\,dx.\]
Similarly, we say that $(u, p)\in \rY^1_2(\Omega)^d\times L_2(\Omega)$ is a weak solution to 
\[\left\{
\begin{aligned}
\cL^* u+\nabla p =f+D_\alpha f_\alpha &\quad \text{in }\, \Omega,\\
\divg u=g &\quad \text{in }\, \Omega,
\end{aligned}
\right.\]
in an unbounded domain $\Omega=\bR^d$ or $\Omega=\bR_+^d$ if $(u, p)$ satisfies the system in the sense of distributions in $\Omega$.
In particular, for any $\phi\in \rY^1_2(\Omega)^d$ 
\[\int_\Omega A^{\alpha\beta}D_\beta \phi \cdot D_\alpha u\,dx - \int_\Omega p \divg \phi\,dx=\int_\Omega f\cdot \phi\,dx-\int_\Omega f_\alpha\cdot D_\alpha \phi\,dx.\]
\end{definition}

\begin{definition}[Green functions on unbounded domains $\Omega$]
\label{D23}
Let $V(x,y)$ be a $d \times d$ matrix valued function and $\Pi(x,y)$ be a $d \times 1$ vector valued function on $\Omega \times \Omega$.
We say that a pair $(V(x,y),\Pi(x,y))$ is the Green function for the Stokes system if it satisfies the following properties.
\begin{enumerate}[$(a)$]
\item
For any $y\in \Omega$, 
$V(\cdot,y)\in W^1_{1,\operatorname{loc}}(\overline{\Omega})^{d\times d}$ and $\Pi(\cdot,y)\in L_{1,\operatorname{loc}}(\overline{\Omega})^d$.
Moreover, $(1-\eta)V(\cdot,y)\in \rY^1_2(\Omega)^{d\times d}$ for all $\eta\in C^\infty_0(\Omega)$ satisfying $\eta\equiv 1$ on $B_r(y)$, where $0<r<d_y$.
\item
For any $y\in \Omega$, $(V(\cdot,y),\Pi(\cdot,y))$ satisfies 
\begin{equation}	\label{160822@eq1}
\divg V(\cdot,y)=0 \quad \text{in }\, \Omega
\end{equation}
and
\[\cL V(\cdot,y) + \nabla \Pi(\cdot,y)=\delta_y\vec I \quad \text{in }\, \Omega\]
in the sense that for any $\phi\in C^\infty_0(\Omega)^d$, we have 
\begin{equation}		\label{160821@eq5}
\int_{\Omega} A^{\alpha\beta}D_\beta  V^{\cdot k}(\cdot,y)\cdot D_\alpha \phi\,dx - \int_{\Omega}\Pi^k(\cdot,y)\divg \phi\,dx=\phi^k(y),
\end{equation}
where $V^{\cdot k}(x,y)$ is the $k$-th $(k\in \{1,\ldots,d\})$ column of $V(x,y)$.
\item
Suppose that  $f\in C^\infty_c(\Omega)^n$ and $g\in C^\infty_c(\Omega)$.
If 
$(u,p)\in \rY^1_2(\Omega)^d\times L_2(\Omega)$ is a weak solution to 
\[
\left\{
\begin{aligned}
\cL^* u+\nabla p =f &\quad \text{in }\, \Omega,\\
\divg u=g &\quad \text{in }\, \Omega,
\end{aligned}
\right.
\]
then 
\[
u(y)=\int_{\Omega}V(\cdot,y)^{\operatorname{tr}} f\,dx
 - \int_{\Omega}\Pi(\cdot,y)g\,dx.
\]
\end{enumerate}
The Green function for the adjoint Stokes system is defined similarly,
and the Green function in $\Omega = \bR^d$ is called the fundamental solution.
We point out that the condition (c) in the above definition gives the uniqueness of a Green function.
\end{definition}

Before stating our main theorems, we introduce the following assumption.
It is known that if the coefficients are $\VMO$ (vanishing mean oscillations), then Assumption \ref{ASSA} holds; see \cite{arXiv:1503.07290v3}.
For more examples of the coefficients satisfying Assumption \ref{ASSA}, see Theorem \ref{MRD}.

\begin{assumption}		
\label{ASSA}
There exist positive real numbers $R_0$, $C_0$, and $\alpha_0 < 1$ such that if $(u,p)\in W^1_2(B_R(x^0))^d\times L_2(B_R(x^0))$ satisfies, in the sense of distributions,
\begin{equation}		\label{161121@EQ1}
\cL u+\nabla p =0,\quad  \divg u=0 \quad \text{in }\, B_R(x^0),
\end{equation}
for some $x^0\in \Omega$ and $0<R\le \min\{R_0,\operatorname{dist}(x^0,\partial \Omega)\}$, then 
\begin{equation*}	
[u]_{C^{\alpha_0}(B_{R/2}(x^0))}\le C_0R^{-\alpha_0}\left(\fint_{B_R(x^0)}|u|^2\,dx\right)^{1/2},
\end{equation*}
where $[u]_{C^{\alpha_0}}$ denotes the usual H\"older seminorm.
The same estimate holds true when $\cL$ is replaced by $\cL^*$.
\end{assumption}

\begin{theorem}
\label{MRA}
Let $\Omega=\bR^d$, $d\ge 3$.
If Assumption \ref{ASSA} holds true, then there exists a unique fundamental solution $(V(x,y), \Pi(x,y))$ for the Stokes problem in  $\Omega$.
Moreover, for any $x,\, y\in \Omega$ satisfying $0<|x-y|\le R_0$, 
\begin{equation}		\label{160822_eq1}
|V(x,y)|\lesssim_{d,\lambda,C_0,\alpha_0} |x-y|^{2-d}.
\end{equation}
Furthermore, if for some $q_0>d$ 
\begin{equation}
\label{170307@eq2}
\begin{split}
f&\in L_{2d/(d+2)}(\Omega)^d\cap L_{q_0/2, \operatorname{loc}}(\Omega)^d, \\
f_\alpha&\in L_2(\Omega)^d\cap L_{q_0,\operatorname{loc}}(\Omega)^d, \\
g&\in L_2(\Omega)\cap L_{q_0,\operatorname{loc}}(\Omega),
\end{split}
\end{equation}
and $(u,p)\in \rY^1_2(\Omega)^d\times L_2(\Omega)$ is a weak solution to 
\begin{equation}		\label{170307@eq3}
\left\{
\begin{aligned}
\cL^* u+\nabla p =f+D_\alpha f_\alpha &\quad \text{in }\, \Omega,\\
\divg u=g &\quad \text{in }\, \Omega,
\end{aligned}
\right.
\end{equation}
then 
\begin{equation}		\label{170307@eq1}
u(y)=\int_{\Omega}V(\cdot,y)^{\operatorname{tr}} f\,dx
-\int_{\Omega}D_\alpha V(\cdot,y)^{\operatorname{tr}}f_\alpha\,dx - \int_{\Omega}\Pi(\cdot,y)g\,dx.
\end{equation}
\end{theorem}

Our next result is about the existence of the Green function for the Stokes system on $\bR^d_+$.
We denote $d_x=\operatorname{dist}(x,\partial \bR_+^d)$ for $x\in \bR_+^d$.

\begin{theorem}		
\label{MRB}
Let $\Omega=\bR^d_+$, $d\ge 3$.
If Assumption \ref{ASSA} holds, then there exists a unique Green function $(V(x,y), \Pi(x,y))$ for the Stokes operator in $\Omega$.
Moreover, for any $x,\, y\in \bR^d_+$ satisfying  $0<|x-y|\le  \min\{d_x, d_y, R_0\}$,  we have 
\[|V(x,y)|\lesssim_{d,\lambda,C_0,\alpha_0} |x-y|^{2-d}.\]
Furthermore, the representation formula \eqref{170307@eq1} is valid.
\end{theorem}

Actually, we will obtain the following corollary in the middle of the proofs of the previous theorems.
But, we record it here to place useful information together.

\begin{corollary}		
\label{MRC}
Let $\Omega=\bR^d$ or $\Omega=\bR_+^d$.
The Green functions constructed in Theorem \ref{MRA} and Theorem \ref{MRB} satisfy the following estimates:
for any $y\in \Omega$ and $0<R\le \min\set{R_0, d_y}$
\begin{enumerate}[$i)$]
\item
$\|V(\cdot,y)\|_{Y^1_2(\Omega\setminus B_R(y))} \lesssim_{d,\lambda,C_0,\alpha_0} R^{1-d/2},$
\item
$\|V(\cdot,y)\|_{L_q(B_R(y))} \lesssim_{d,\lambda,C_0,\alpha_0,q} R^{2-d+d/q}, \quad q\in [1,d/(d-2)),$
\item
$\|DV(\cdot,y)\|_{L_q(B_R(y))} \lesssim_{d,\lambda,C_0,\alpha_0,q} R^{1-d+d/q}, \quad q\in[1,d/(d-1)),$
\item
$\|\Pi(\cdot,y)\|_{L_2(\Omega\setminus B_R(y))} \lesssim_{d,\lambda,C_0,\alpha_0} R^{1-d/2},$
\item
$\|\Pi(\cdot,y)\|_{L_q(B_R(y))} \lesssim_{d,\lambda,C_0,\alpha_0,q} R^{1-d+d/q}, \quad q\in[1,d/(d-1))$.
\end{enumerate}
\end{corollary}		

\begin{remark}
\label{R11}
Theorem \ref{MRA}, Theorem \ref{MRB}, and Corollary \ref{MRC} continue to hold for the adjoint system under Assumption \ref{ASSA}.
\end{remark}

\begin{corollary}		
\label{161122@cor1}
Let $\Omega=\bR^d$ or $\Omega=\bR_+^d$.
Let $({}^*V(x,y),{}^*\Pi(x,y))$ be the Green function for the adjoint problem.
Then for $x\neq y$ 
\begin{equation}		\label{161120@eq1}
V(x,y)={}^*V(y,x)^{\operatorname{tr}}.
\end{equation}
Moreover, if $(u, p) \in \rY^1_2(\Omega)^d\times L_2(\Omega)$ satisfies 
\[\left\{
\begin{aligned}
\cL u+\nabla p =f+D_\alpha f_\alpha &\quad \text{in }\, \Omega,\\
\divg u=0 &\quad \text{in }\, \Omega,
\end{aligned}
\right.\]
with  \eqref{170307@eq2}, then 
\begin{equation}		\label{161122@eq8}
u(y)=\int_{\Omega} V(y,\cdot)f\,dx-\int_{\Omega} D_\alpha V(y,\cdot)f_\alpha\,dx.
\end{equation}
\end{corollary}

\begin{remark}		\label{0717@rem1}
When $\cL=\cL^*$, i.e., $A^{\alpha\beta}_{ij}=A^{\beta\alpha}_{ji}$, we have $V(x,y)=V(y,x)^{\operatorname{tr}}$ from \eqref{161120@eq1}.
\end{remark}

The following theorem shows some examples satisfying Assumption \ref{ASSA}.

\begin{theorem}
\label{MRD}
\begin{enumerate}[(a)]
\item
If the coefficients $A^{\alpha\beta}$ of $\cL$ are merely measurable functions of only  one fixed direction, i.e.,
\[A^{\alpha\beta}=A^{\alpha\beta}(x_k) \, \text{ for some }  k\in \{1,\ldots,d\},\]
then for any $\alpha_0\in (0,1)$ and $R_0\in (0,\infty)$, Assumption \ref{ASSA} holds with $C_0=C_0(d,\lambda,\alpha_0)$.
\item
Let $\alpha_0\in (0,1)$.
There exists a constant $\gamma\in (0,1)$, depending on $d$, $\lambda$, and $\alpha_0$, such that if 
\[\sup_{x\in \bR^d}\sup_{r\le R_0}\fint_{B_r(x)}\bigg|A^{\alpha\beta}(y_1,y')-\fint_{B_r'(x')}A^{\alpha\beta}(y_1,z')\,dz'\bigg|\,dy\le \gamma,\]
for some $R_0\in (0,\infty)$, then  Assumption \ref{ASSA} holds with $C_0=C_0(d,\lambda,\alpha_0)$.
The statement remains true, provided that $y_1$ and $y'$ are replaced by $y_k$ and $(y_1,\ldots,y_{k-1},y_{k+1},\ldots,y_d)$, respectively.
\end{enumerate}
\end{theorem}

Next we consider the pointwise bound for the Green function on half space under the additional assumption.

\begin{assumption}		
\label{ASSB}
There exist positive numbers $R_1$ and $C_1$ such that if $(u,p)\in W^1_2(\bR^d_+\cap B_R(x^0))^d\times L_2(\bR^d_+\cap B_R(x^0))$ satisfies
\begin{equation}		\label{161025@eq1}
\left\{
\begin{aligned}
&\cL u+\nabla p =0, \quad \divg u=0 \quad \text{in }\, \bR^d_+\cap B_R(x^0),\\
&u=0 \quad \text{on }\, \partial\bR^d_+ \cap B_R(x^0),
\end{aligned}
\right.
\end{equation}
for some $x^0\in \partial\bR^d_+$ and $0<R\le R_1$, then 
\begin{equation}		\label{161019@aa1}
\|u\|_{L_\infty(\bR^d_+\cap B_{R/2}(x^0))}\le C_1\left(\frac{1}{R^d}\int_{\bR^d_+\cap B_R(x^0)}|u|^2\,dx\right)^{1/2}.
\end{equation}
The same estimate holds true if $\cL$ is replaced by $\cL^*$.
\end{assumption}

\begin{theorem}		
\label{MRE}
Suppose that  Assumptions \ref{ASSA} and \ref{ASSB} hold.
Let $(V(x,y),\Pi(x,y))$ be the Green function constructed in Theorem \ref{MRB}.
Then for any $x,\, y\in \bR^d_+$ satisfying $0<|x-y|\le \min\{R_0,R_1\}$, 
\begin{equation}		\label{161019@eq5}
|V(x,y)|\lesssim_{d,\lambda,C_0,\alpha_0,C_1} |x-y|^{2-d}.
\end{equation}
Moreover, for any $y\in \bR^d_+$ and $0<R\le \min\{R_0,R_1\}$, 
\begin{enumerate}[$i)$]
\item
$\|V(\cdot,y)\|_{Y^1_2(\bR^d_+\setminus B_R(y))} \lesssim_{d,\lambda,C_0,\alpha_0,C_1} R^{1-d/2}$,
\item
$\|V(\cdot,y)\|_{L_q(\bR^d_+\cap B_R(y))} \lesssim_{d,\lambda,C_0,\alpha_0,C_1,q} R^{2-d+d/q}, \quad q\in [1,d/(d-2)),$
\item
$\|DV(\cdot,y)\|_{L_q(\bR^d_+\cap B_R(y))} \lesssim_{d,\lambda,C_0,\alpha_0,C_1,q} R^{1-d+d/q}, \quad q\in[1,d/(d-1)),$
\item
$\|\Pi(\cdot,y)\|_{L_2(\bR^d_+\setminus B_R(y))} \lesssim_{d,\lambda,C_0,\alpha_0,C_1} R^{1-d/2},$
\item
$\|\Pi(\cdot,y)\|_{L_q(\bR^d_+\cap B_R(y))} \lesssim_{d,\lambda,C_0,\alpha_0,C_1,q} R^{1-d+d/q}, \quad q\in[1,d/(d-1))$.
\end{enumerate}
\end{theorem}

The following theorem shows some examples satisfying Assumption \ref{ASSB}.

\begin{theorem}
\label{MRF}
\begin{enumerate}[(a)]
\item
If the coefficients $A^{\alpha\beta}$ of $\cL$ are merely measurable functions of only  $x_1$-direction, i.e.,
\[A^{\alpha\beta}=A^{\alpha\beta}(x_1),\]
then for any $R_1\in (0,\infty)$ Assumption \ref{ASSB} holds for some $C_1=C_1(d,\lambda)$.
\item
There exists a number $\gamma\in (0,1)$, depending on $d$ and $\lambda$, such that if 
\[\sup_{x\in \bR^d}\sup_{r\le R_1}\fint_{B_r(x)}\bigg|A^{\alpha\beta}(y_1,y')-\fint_{B_r'(x')}A^{\alpha\beta}(y_1,z')\,dz'\bigg|\,dy\le \gamma,\]
for some $R_1\in (0,\infty)$,
then Assumption \ref{ASSB} holds for some $C_1=C_1(d,\lambda)$.
\end{enumerate}
\end{theorem}

The following assumption is used to obtain a better estimate for the Green function near the boundary.

\begin{assumption}		\label{ASSC}
There exist positive real numbers $R_2$, $C_2$, and $\alpha_2<1$ such that if  $(u,p)\in W^1_2(\bR^d_+\cap B_R(x^0))^d\times L_2(\bR^d_+\cap B_R(x^0))$ satisfies, in the sense of distributions,
\begin{equation}		\label{170311@eq1}
\left\{
\begin{aligned}
&\cL u+\nabla p =0, \quad \divg u=0 \quad \text{in }\, \bR^d_+\cap B_R(x^0),\\
&u=0 \quad \text{on }\, \partial \bR^d_+\cap  B_R(x^0),
\end{aligned}
\right.
\end{equation}
for some $x^0\in \bR^d_+$ and $0<R\le R_2$, then 
\begin{equation*}		
\big[u\chi_{\bR^d_+\cap B_{R}(x^0)}\big]_{C^{\alpha_2}(B_{R/2}(x^0))}\le C_2R^{-\alpha_2}\left(\fint_{\bR^d_+\cap B_R(x^0)}|u|^2\,dx\right)^{1/2}.
\end{equation*}
The same estimate holds true when $\cL$ is replaced by $\cL^*$.
\end{assumption}

\begin{remark}
It will be clear from the proof of Theorem \ref{MRF} that Assumption \ref{ASSC} holds under the hypothesis in $(a)$ or $(b)$ of Theorem \ref{MRF}. 
\end{remark}

We observe that  Assumption \ref{ASSC} implies Assumptions \ref{ASSA} and \ref{ASSB}.
By Theorem \ref{MRE}, under Assumption \ref{ASSC}, there exists the Green function $(V(x,y),\Pi(x,y)$ for the Stokes problem satisfying the pointwise estimate \eqref{161019@eq5} in Theorem \ref{MRE}.
The following theorem shows that a better estimate for $V(x,y)$  is available near the boundary $\partial \bR^d_+$. 
We denote $d_x=\operatorname{dist}(x,\partial \bR_+^d)$ for $x\in \bR_+^d$.

\begin{theorem}		\label{MRG}
Suppose that  Assumption \ref{ASSC} holds.
Let $(V(x,y),\Pi(x,y))$ be the Green function constructed in Theorem \ref{MRB}.
Then for any $x,\, y\in \bR^d_+$ with $x\neq y$,
\begin{equation}		\label{170304@eq1}
|V(x,y)|\le C\min\{d_x,|x-y|,R_2\}^{\alpha_2}\min\{d_y,|x-y|,R_2\}^{\alpha_2} \min\{|x-y|,R_2\}^{2-d-2\alpha_2},
\end{equation}
where $C=C(d,\lambda,C_2,\alpha_2)$.
\end{theorem}

In a bounded Lipschitz domain, the estimate \eqref{170304@eq1} of the Green function for the classical Stokes system with the Laplace operator was proved by Chang-Choe \cite{MR2465713} and Kang-Kim \cite{MR2718661}.
In particular, \cite{MR2718661} dealt with the estimate \eqref{170304@eq1} of the Green functions for elliptic systems with irregular coefficients.

\section{Auxiliary lemmas}
\label{S4}

In this section, we review the existence of solutions to the divergence equation.
We also gather some auxiliary lemmas about unique solvability results, pressure estimates, and gradient estimates for the Stokes system with measurable coefficients in the whole space and half space.

\begin{lemma}		\label{161121@lem3}
Let $1<q<\infty$.
\begin{enumerate}[$(a)$]
\item
Let $\Omega$ be a bounded Lipschitz domain in $\bR^d$.
Then for any  $g\in L_q(\Omega)$ satisfying $(g)_\Omega=0$, there exists $u\in \mathring{W}^1_q(\Omega)^d$ such that 
\begin{equation*}		
\divg u=g \ \text{ in }\ \Omega, \quad \|Du\|_{L_q(\Omega)} \lesssim_{d,q,\Lip(\Omega)} \|g\|_{L_q(\Omega)}
\end{equation*}
where $\Lip(\Omega)$ denotes the Lipschitz constant of $\Omega$.
\item
Let $\Omega=B_R$.
Then for any  $g\in L_q(\Omega)$ satisfying $(g)_\Omega=0$, there exists $u\in \mathring{W}^1_q(\Omega)^d$ such that 
\[\divg u=g \ \text{ in }\ \Omega, \quad \|Du\|_{L_q(\Omega)}\lesssim_{d,q} \|g\|_{L_q(\Omega)}.\]
This remains true when $B_R$ is replaced by $B_R^+$, $B_R\setminus \overline{B_{R/2}}$, or $B_R^+\setminus \overline{B_{R/2}^+}$.
\end{enumerate}
\end{lemma}

\begin{proof}
For the proof of $(a)$ we refer to \cite{MR2101215}.
Using $(a)$ and scaling, one can show  $(b)$.
\end{proof}

The problem of the existence of solutions to the divergence equation in various domains $\Omega$ has been studied by many authors upon the regularity assumptions made on $\Omega$ and the construction methods of solutions $u$.
We note that the existence of solutions to the divergence equation in the whole space and half space can be deduced from Lemma \ref{161121@lem3} with scaling; see also \cite[p. 261, Corollary IV.3.1]{MR2808162}.
For the half space case, there is a method based on some explicit representation formula, wihch was studied in detail by Cattabriga \cite{Cattabriga} and Solonnikov \cite{Solonnikov}.

\begin{lemma}		\label{160810@lem1}
Let $\Omega=\bR^d$ or $\bR^d_+$.
If $1<q<d$ and $g\in L_q(\Omega)$, 
then there exists $u\in \rY^1_q(\Omega)^d$ such that 
\begin{equation*}		
\divg u=g  \ \text{ in }\ \Omega, \quad \|Du\|_{L_q(\Omega)}\lesssim_{d,q}\|g\|_{L_q(\Omega)}.
\end{equation*} 
\end{lemma}

\begin{lemma}		\label{161121@lem1}
Let $\Omega=\bR^d$ or $\bR^d_+$.
Then for $f\in L_{2d/(d+2)}(\Omega)^d$, $f_\alpha\in L_2(\Omega)^d$, and $g\in L_2(\Omega)$, there exists a unique weak solution $(u,p)\in \rY^1_2(\Omega)^d\times L_2(\Omega)$ to the problem
\begin{equation*}		
\left\{
\begin{aligned}
\cL u+\nabla p=f+D_\alpha f_\alpha &\quad \text{in }\, \Omega,\\
\divg u=g &\quad \text{in }\, \Omega.
\end{aligned}
\right.
\end{equation*}
Moreover,
\begin{equation}		\label{160721@eq2}
\|Du\|_{L_2(\Omega)}+\|p\|_{L_2(\Omega)} \lesssim_{d,\lambda} \|f\|_{L_{2d/(d+2)}(\Omega)}+\|f_\alpha\|_{L_2(\Omega)}+\|g\|_{L_2(\Omega)}.
\end{equation}
\end{lemma}

\begin{proof}
The proof is based on Lemma \ref{160810@lem1} and the Lax-Milgram theorem.
We omit the proof because it is almost the same as that of \cite[Lemma 3.1]{arXiv:1503.07290v3}.
\end{proof}

\begin{lemma}		\label{160808@lem2}
Let $R>0$. If $(u,p)\in W^1_2(B_R)^d\times L_2(B_R)$ satisfies
\[\cL u+\nabla p=0 \quad \text{in }\, B_R,\]
then 
\[\int_{B_R}|p-(p)_{B_{R}}|^2\,dx \lesssim_{d,\lambda} \int_{B_R}|Du|^2\,dx.\]
The same estimate holds true if $B_R$ is replaced by $B_R^+$, $B_R\setminus \overline{B_{R/2}}$, or $B_R^+\setminus \overline{B_{R/2}^+}$.
\end{lemma}

\begin{proof}
The proof is almost the same as the classical case.
For reader's conveneicne we sketch the proof. 
From the solvability of the divergence equation, there exists $\phi\in \mathring{W}^1_2(B_R)^d$ such that 
\[\divg \phi=p-(p)_{B_R}  \ \text{ in }\ B_R, \quad \|D\phi\|_{L_2(B_R)} \lesssim_d \|p-(p)_{B_R}\|_{L_{2}(B_R)}.\]
Using $\phi$ as a test function we obtain 
\[\|p-(p)_{B_R}\|_{L_{2}(B_R)}^2 = \int (p-(p)_{B_R}) \divg \phi = \int \cL u \cdot \phi.\]
The result follows from the strong ellipticity condition with the Cauchy inequality.
\end{proof}

\begin{lemma}		\label{160920@lem2}
Let $R>0$.
\begin{enumerate}[$(a)$]
\item
If $(u,p)\in W^1_2(B_R)^d\times L_2(B_R)$ satisfies the system
\[\left\{
\begin{aligned}
\cL u+\nabla p=0 &\quad \text{in }\, B_R,\\
\divg u=0 &\quad \text{in }\, B_R,
\end{aligned}
\right.\]
then we have
\[\int_{B_{R/2}}|Du|^2\,dx \lesssim_{d,\lambda} R^{-2}\int_{B_R}|u|^2\,dx.\]
The statement remains true, provided that $B_R$ and $B_{R/2}$ are replaced by $B_{5R/4}\setminus \overline{B_{R/4}}$ and $B_R\setminus \overline{B_{R/2}}$, respectively.

\item
If $(u, p)\in W^1_2(B_R^+)^d\times L_2(B_R^+)$ satisfies the system 
\[\left\{
\begin{aligned}
\cL u+\nabla p=0 &\quad \text{in }\, B_R^+,\\
\divg u=0 &\quad \text{in }\, B_R^+,\\
u=0 &\quad \text{on }\,  B_R \cap \partial \bR^d_+,
\end{aligned}
\right.\]
then we have 
\[
\int_{B_{R/2}^+}|Du|^2\,dx \lesssim_{d,\lambda} R^{-2}\int_{B_{R}^+}|u|^2\,dx.
\]
The statement remains true, provided that $B_R^+$, $B_{R/2}^+$, and $B_R$ are replaced by $B_{5R/4}^+\setminus \overline{B_{R/4}}$, $B_R^+\setminus \overline{B_{R/2}}$,  and $B_{5R/4}\setminus \overline{B_{R/4}}$, respectively.

\end{enumerate}
\end{lemma}

\begin{proof}
For a proof, one can just refer to the proofs of \cite[Lemma 3.2]{MR2027755} and  \cite[Lemma 3.6]{arXiv:1604.02690v2} with obvious modifications.
For reader's convenience we sketch the proof for the case 
when $(u,p)\in W^1_2\big(B^+_{5R/4}\setminus \overline{B_{R/4}}\big)^d\times L_2\big(B^+_{5R/4}\setminus \overline{B_{R/4}}\big)$
 in $(b)$. 

We denote for $r>0$ 
\[\cC_r=B_{R+r}\setminus B_{\frac{R}{2}-r} \quad \text{and} \quad \cC_r^+=B_{R+r}^+\setminus B_{\frac{R}{2}-r}^+.\]
Let $0<\rho<r\le R/4$ and $\eta$ be a smooth function on $\bR^d$ satisfying
\[0\le \eta\le 1, \quad \eta=1\,\text{ on }\, \cC_\rho, \quad \operatorname{supp}\eta\subset \cC_r, \quad |D\eta| \lesssim (r-\rho)^{-1}.\]
Using $\eta^2u$ as a test function to
\[\cL u+\nabla p=0 \quad \text{in }\, \cC_r^+,\]
we obtain the Caccioppoli type inequality; for all $\varepsilon>0$
\[\int_{\cC_\rho^+} |Du|^2\,dx\le \varepsilon\int_{\cC_r^+}|p-(p)_{\cC_r^+}|^2\,dx + \frac{C(d,\lambda, \varepsilon)}{ (r-\rho)^{2}}\int_{\cC_r^+}|u|^2\,dx.\]
Using the pressure estimate, Lemma \ref{160808@lem2}, we have for all $0<\rho<r \le \frac{R}{4}$
\begin{equation}		\label{160921@eq1}
\int_{\cC_\rho^+}|Du|^2\,dx\le \varepsilon \int_{\cC_r^+}|Du|^2\,dx + \frac{C}{(r-\rho)^2}\int_{\cC_r^+}|u|^2\,dx.
\end{equation}
For $k=0,1,2,\ldots$ we set 
\[\varepsilon=\frac{1}{8}, \quad r_k=\frac{R}{4}\left(1-\frac{1}{2^k}\right)\]
so that \eqref{160921@eq1} becomes 
\[\int_{\cC_{r_k}^+}|Du|^2\,dx\le \varepsilon \int_{\cC_{r_{k+1}}^{+}}|Du|^2\,dx + \frac{C4^k}{R^2}\int_{\cC_{r_{k+1}}^{+}}|u|^2\,dx.\]
Multiplying $\varepsilon^k$ and summing the estimates we obtain the required result.
\end{proof}

\begin{lemma}		\label{161121@lem7}
\begin{enumerate}[$(a)$]
\item
Let Assumption \ref{ASSA} hold. 
If
$(u,p)\in W^1_2(B_R(x^0))^d\times L_2(B_R(x^0))$ satisfies \eqref{161121@EQ1} with $x^0\in \bR^d$ and $0<R\le R_0$, then 
\[\|u\|_{L_\infty(B_{R/2}(x^0))} \lesssim_{d,C_0,\alpha_0} R^{-d}\|u\|_{L_1(B_R(x^0))}.\]
\item
Let Assumptions \ref{ASSA} and \ref{ASSB} hold. 
If
$(u, p)\in W^1_2(B_R^+(x^0))^d\times L_2(B_R^+(x^0))$ satisfies \eqref{161025@eq1} with $x^0\in \partial \bR^d_+$ and $0<R\le \min\{R_0,R_1\}$, then 
\[\|u\|_{L_\infty(B_{R/2}^+(x^0))} \lesssim_{d,C_0,\alpha_0,C_1} R^{-d}\|u\|_{L_1(B_R^+(x^0))}.\]
\end{enumerate}
\end{lemma}

\begin{proof}
We only prove the second assertion of the lemma because the first one is the same with obvious modifications.
Let $0<r<R$ and set $\rho=\frac{R-r}{8}$.
We can choose $y^0\in B_r^+(x^0)$ satisfying 
\[\frac{1}{2} \sup_{B_r^+(x^0)} |u|^2 \le \sup_{B_\rho(y^0)\cap \bR^d_+} |u|^2.\]
If $2\rho\le \operatorname{dist}(y^0,\partial \bR^d_+)$, then by Assumption \ref{ASSA}
\[\sup_{B_\rho(y^0)} |u|^2 \lesssim \fint_{B_{2\rho}(y^0)} |u|^2\,dx \lesssim (R-r)^{-d} \int_{B_{R}^+(x^0)} |u|^2\,dx.\]
On the other hand, if $2\rho>\operatorname{dist}(y^0,\partial \bR^d_+)$, then by Assumption \ref{ASSB}
\[\sup_{B_\rho(y^0)\cap \bR^d_+} |u|^2\lesssim \sup_{B_{4\rho}^+(z^0)}|u|^2 \lesssim \fint_{B_{8\rho}^+(z^0)} |u|^2\,dx \lesssim (R-r)^{-d} \int_{B_{R}^+(x^0)} |u|^2\,dx,\]
where $z^0=(0,y^0_2,\ldots,y^0_d)$.
Hence Young's inequality yields that for $0<r<R$ and $\varepsilon>0$
\begin{align*}
\sup_{B_r^+(x^0)} |u|
&\lesssim_{d,C_0,\alpha_0,C_1} (R-r)^{-d/2}\|u\|_{L_2(B_R^+(x^0))} \\
&\le \varepsilon \sup_{B_R^+(x^0)} |u| + C_\varepsilon (R-r)^{-d}\|u\|_{L_1(B_R^+(x^0))}. 
\end{align*}
Now, the result follows from a standard iteration argument in \cite[pp. 80--82]{MR1239172}. 
\end{proof}

\section{Proof of Theorem \ref{MRA}}
\label{S5}

The proof is a modification of the argument for elliptic systems found in Hofmann--Kim \cite[Theorem 3.1]{MR2341783}.
Throughout this section, $R_0$, $C_0$, and $\alpha_0$ are constants in Assumption \ref{ASSA}, and we divide the proof into several steps.

\begin{enumerate}[\bf{Step} 1)]
\item
First we define an averaged fundamental solution on $\bR^d$ as follows.
For each $\varepsilon>0$, $y\in \bR^d$, and $k\in \{1,\ldots,d\}$ we denote 
\[f_{\varepsilon;y,k} 
= \frac{\chi_{B_\varepsilon(y)}}{|B_\varepsilon(y)|}e_k\]
where $\chi_{B_\varepsilon(y)}$ is the characteristic function and $e_k$ is the $k$-th unit vector in $\bR^d$.
By Lemma \ref{161121@lem1} there is a unique weak solution $(v_{\varepsilon;y,k},\pi_{\varepsilon;y,k}) \in Y^1_2(\bR^d)^d\times L_2(\bR^d)$ to
\[\left\{
\begin{aligned}
\cL v + \nabla \pi = f_{\varepsilon;y,k}  &\quad \text{in }\, \bR^d,\\
\divg v = 0 &\quad \text{in }\, \bR^d.
\end{aligned}
\right.\]
We define \emph{the averaged fundamental solution} $(V_\varepsilon(\cdot,y), \Pi_\varepsilon(\cdot,y))$ by 
\[V_\varepsilon^{jk}(\cdot,y)=v_{\varepsilon;y,k}^j \quad \text{and}\quad \Pi_\varepsilon^k(\cdot,y)=\pi_{\varepsilon;y,k}.\]
Hereafter, we denote by $V_\varepsilon^{\cdot k}(x,y)$ the $k$-th column of $V_\varepsilon(x,y)$.
Then 
\begin{equation}		\label{160727@eq1}
\int_{\bR^d} A^{\alpha\beta}D_\beta V_\varepsilon^{\cdot k}(\cdot,y)\cdot D_\alpha \phi\,dx - \int_{\bR^d}\Pi_\varepsilon^k(\cdot,y)\divg  \phi\,dx=\fint_{B_{\varepsilon}(y)}\phi^k\,dx
\end{equation}
for all $\phi \in C^\infty_0(\bR^d)^d$.
Moreover, from \eqref{160721@eq2}, 
\begin{equation}		\label{160727@eq1a}
\|DV_\varepsilon(\cdot,y)\|_{L_2(\bR^d)}+\|\Pi_\varepsilon(\cdot,y)\|_{L_2(\bR^d)}\lesssim \varepsilon^{1-d/2}.
\end{equation}
\item
We prove the local pointwise estimate for $V_\varepsilon(x,y)$.
\begin{lemma}		
\label{L51}
If Assumption \ref{ASSA} holds, then 
\[|V_\varepsilon(x,y)| \lesssim_{d,\lambda,C_0,\alpha_0} |x-y|^{2-d}\]
for all $x,\,y\in \bR^d$ and $\varepsilon>0$ satisfying $0<\varepsilon\le|x-y|/3\le R_0/2$.
\end{lemma}

\begin{proof}
Let 
\[0<\varepsilon\le R:=\frac{|x-y|}{3}\le \frac{R_0}{2}.\]
Since $(V_\varepsilon(\cdot,y), \Pi_\varepsilon(\cdot,y))$ satisfies 
\[\left\{
\begin{aligned}
\cL V_\varepsilon^{\cdot k}(\cdot,y) + \nabla\Pi^k_\varepsilon(\cdot,y)=0 &\quad \text{in }\, B_R(x),\\
\divg  V_\varepsilon^{\cdot k}(\cdot,y)=0 &\quad \text{in }\, B_R(x),
\end{aligned}
\right.\]
By Lemma \ref{161121@lem7} 
$$
|V_\varepsilon^{\cdot k}(x,y)|\lesssim_{d,C_0,\alpha_0} R^{-d}\|V_\varepsilon^{\cdot k}(\cdot,y)\|_{L_1(B_{R}(x))}.
$$
Thus, it suffices to show that 
\begin{equation}		\label{161018@eq1a}
\|V_\varepsilon^{\cdot k}(\cdot,y)\|_{L_1(B_R(x))}\lesssim R^2.
\end{equation}
Let $f\in L_\infty(\bR^d)^d$ with $\operatorname{supp}f\subset B_R(x)$ and $(u,p)\in {Y}^1_2(\bR^d)^d\times L_2(\bR^d)$ be the weak solution  to 
\[\left\{
\begin{aligned}
\cL^* u + \nabla p=f &\quad \text{in }\, \bR^d,\\
\divg u=0 &\quad \text{in }\, \bR^d.
\end{aligned}
\right.\]
By testing with $V_\varepsilon^{\cdot k}(\cdot,y)$ in the above system, \[\int_{\bR^d}A^{\alpha\beta}D_\beta V_\varepsilon^{\cdot k}(\cdot,y)\cdot D_\alpha u\,dx=\int_{B_R(x)}V_\varepsilon^{\cdot k}(\cdot,y)\cdot f\,dx.\]
Also, by testing with $\phi=u$ in \eqref{160727@eq1}, 
\[\int_{\bR^d}A^{\alpha\beta}D_\beta V_\varepsilon^{\cdot k}(\cdot,y)\cdot D_\alpha u\,dx=\fint_{B_\varepsilon(y)}u^k\,dx.\]
Hence 
\begin{equation}		\label{160822@eq3}
\int_{B_R(x)}V_\varepsilon^{\cdot k}(\cdot,y)\cdot f\,dx = \fint_{B_\varepsilon(y)}u^k\,dx.
\end{equation}
Since $(u,p)$ satisfies
\[\left\{
\begin{aligned}
\cL^* u + \nabla p=0 &\quad \text{in }\, B_{2R}(y),\\
\divg u=0 &\quad \text{in }\, B_{2R}(y),
\end{aligned}
\right.\]
we use Lemma \ref{161121@lem7}, H\"older's inequality, and the Sobolev inequality to obtain 
\[\|u\|_{L_\infty(B_{R}(y))} \lesssim R^{1-d/2}\|u\|_{L_{2d/(d-2)}(\bR^d)}\lesssim R^{1-d/2}\|Du\|_{L_{2}(\bR^d)}.\]
Thus, from the estimate \eqref{160721@eq2} we conclude that 
\[\|u\|_{L_\infty(B_{R}(y))}\lesssim_{d,\lambda,C_0,\alpha_0} R^2\|f\|_{L_\infty(B_R(x))}.\]
Using this together with \eqref{160822@eq3} and the duality argument,  we get \eqref{161018@eq1a}.
\end{proof}
\item
We prove the uniform estimates for $V_\varepsilon(\cdot,y)$.

\begin{lemma}		
\label{L52}
If Assumption \ref{ASSA} holds, then for any  $y\in \bR^d$, $0<R\le R_0$, and $\varepsilon>0$
\begin{equation}		\label{160802@eq1}
\|V_\varepsilon(\cdot,y)\|_{Y^1_2(\bR^d\setminus B_R(y))} \lesssim_{d,\lambda,C_0,\alpha_0} R^{1-d/2}.
\end{equation}
\end{lemma}

\begin{proof}
When $\varepsilon\ge R/12$, we have, from \eqref{160727@eq1a} and the Sobolev inequality,
\[\|V_\varepsilon(\cdot,y)\|_{Y^1_2(\bR^d\setminus B_R(y))}\le \|V_\varepsilon(\cdot,y)\|_{Y^1_2(\bR^d)}\lesssim R^{1-d/2}.\]
So, we assume $\varepsilon \in (0, R/12)$.
Denote  $\cD=B_R(y)\setminus \overline{B_{R/2}(y)}$ and let  $\eta$ be a smooth function on $\bR^d$ satisfying 
\[0\le \eta\le 1, \quad \eta\equiv 1 \ \text{ on }\ B_{R/2}(y), \quad \operatorname{supp} \eta\subset B_{R}(y), \quad |D\eta|\lesssim R^{-1}.\]
Then 
\begin{equation}
\label{E50}
\begin{split}
&\|V_{\varepsilon}^{\cdot k}(\cdot,y)\|_{L_{2d/(d-2)}(\bR^d\setminus B_R(y)))} \\
&\le \|(1-\eta^2)V_{\varepsilon}^{\cdot k}(\cdot,y)\|_{L_{2d/(d-2)}(\bR^d)} \\
&\lesssim \|D((1-\eta^2)V^{\cdot k}_{\varepsilon}(\cdot,y))\|_{L_2(\bR^d)} \\
&\lesssim \|(1-\eta^2)DV^{\cdot k}_{\varepsilon}(\cdot,y)\|_{L_2(\bR^d)}+R^{-1}\|V^{\cdot k}_{\varepsilon}(\cdot,y)\|_{L_2(\cD)}.
\end{split}
\end{equation}
We shall show that 
\begin{equation}		
\label{160809@eq1}
\|(1-\eta^2)DV^{\cdot k}_{\varepsilon}(\cdot,y)\|_{L_2(\bR^d)}
\lesssim  R^{-1}\|V^{\cdot k}_{\varepsilon}(\cdot,y)\|_{L_2(\cD_0)}
\end{equation}
where $\cD_0=B_{5R/4}(y)\setminus \overline{B_{R/4}(y)}$.
To show this, we observe first that 
\[\int_{\bR^d}\divg \big((1-\eta^2) V^{\cdot k}_{\varepsilon}(\cdot,y)\big)\,dx=0,\]
so we can subtract an average to get 
\begin{equation}
\label{160807@eq1}
\begin{split}
&\left|\int_{\bR^d}\Pi_{\varepsilon}^k(\cdot,y)\divg \big((1-\eta^2)V^{\cdot k}_{\varepsilon}(\cdot,y)\big)\,dx\right|\\
&=\left|\int_{\bR^d}\big(\Pi_{\varepsilon}^k(\cdot,y) -(\Pi^k_\varepsilon(\cdot,y))_\cD\big)2\eta D\eta \cdot V^{\cdot k}_{\varepsilon}(\cdot,y)\,dx\right|\\
&\lesssim \int_{\cD}\big|\Pi^k_\varepsilon(\cdot,y)-(\Pi^k_{\varepsilon}(\cdot,y))_{\cD}\big|^2\,dx 
+  R^{-2}\int_{\cD}|V_{\varepsilon}^{\cdot k}(\cdot,y)|^2\,dx.
\end{split}
\end{equation}
Using the test function $\phi=(1-\eta^2) V_\varepsilon^{\cdot k}(\cdot,y)$ in \eqref{160727@eq1} and using \eqref{160807@eq1}, we get 
\begin{align*}
&\int_{\bR^d}(1-\eta^2)|DV_\varepsilon^{\cdot k}(\cdot,y)|^2\,dx \\
&\lesssim \int_{\cD}\big|\Pi_{\varepsilon}^k(\cdot,y)-(\Pi^k_\varepsilon(\cdot,y))_{\cD}\big|^2\,dx   
+ R^{-2}\int_{\cD}|V_\varepsilon^{\cdot k}(\cdot,y)|^2\,dx \\
&\quad + \int_{\cD}|DV_\varepsilon^{\cdot k}(\cdot,y)|^2\,dx.
\end{align*}
Thus, using Lemma \ref{160808@lem2} and Lemma \ref{160920@lem2} $(a)$ we get \eqref{160809@eq1}.

Finally, using Lemma \ref{L51} and the fact 
\[0<\varepsilon< \frac{R}{12}< \frac{|x-y|}{3}<\frac{5R}{12}<\frac{R_0}{2}, \quad \forall x\in \cD_0,\]
we have 
\[R^{-2}\int_{\cD_0}|V_\varepsilon^{\cdot k}(\cdot,y)|^2\,dx \lesssim R^{2-d}.\]
Combining this with \eqref{E50} and \eqref{160809@eq1} yields the estimate \eqref{160802@eq1}.
\end{proof}

\item
We prove uniform $L^q$-estimates for $V_\varepsilon(\cdot,y)$ and $DV_\varepsilon(\cdot,y)$.

\begin{lemma} 
\label{L53}
If Assumption \ref{ASSA} holds, then for any  $y\in \bR^d$, $0<R\le R_0$, and $\varepsilon>0$
\begin{align}
\label{160802@eq1a}
\|V_\varepsilon(\cdot,y)\|_{L_q(B_R(y))} \lesssim_{d,\lambda,C_0,\alpha_0,q} R^{2-d+d/q}, \quad q\in [1,d/(d-2)),\\
\label{160802@eq1b}
\|DV_\varepsilon(\cdot,y)\|_{L_q(B_R(y))} \lesssim_{d,\lambda,C_0,\alpha_0,q} R^{1-d+d/q}, \quad q\in [1,d/(d-1)).
\end{align}
\end{lemma}

\begin{proof}
From the previous lemma we have for all $0 < \rho \le R_0$
\[\int_{\bR^d\setminus B_\rho(y)} |V_{\varepsilon}(\cdot,y)|^{2d/(d-2)} \,dx \lesssim \rho^{-d}.\]
Let $0<t<\infty$ and denote 
\[A(t)=\{x\in \bR^d:|V_{\varepsilon}(x,y)|>t\}.\] 
Then for all $0< \rho \le R_0$
\begin{align*}
|A(t)| 
&= |A(t) \cap B_\rho(y)| + |A(t) \setminus B_\rho(y)| \\
&\lesssim \rho^d + t^{-2d/(d-2)} \int_{A(t) \setminus B_\rho(y)} |V_{\varepsilon}(\cdot,y)|^{2d/(d-2)} \,dx \\
&\lesssim \rho^d + t^{-2d/(d-2)} \rho^{-d}.
\end{align*}
If $t \ge R_0^{2-d}$, then we can take $\rho = t^{-1/(d-2)}$ so that 
\[|A(t)| \lesssim t^{-d/(d-2)}.\]
Hence, for all $R_0^{2-d} < T < \infty$
\begin{align*}
\int_{B_R(y)}|V_{\varepsilon}(\cdot,y)|^q\,dx 
&\lesssim \int_0^\infty t^{q-1} |B_R(y) \cap A(t)| \,dt \\
&\lesssim \int_0^T t^{q-1} R^d \,dt + \int_T^\infty t^{q-1} t^{-d/(d-2)} \,dt \\
&\lesssim T^q R^d + T^{q-d/(d-2)}.
\end{align*}
In the last estimate, we have used the condition $q < d/(d-2)$.
If $0< R \le R_0$, then we can take $T = R^{2-d}$ so that 
\[\int_{B_R(y)}|V_{\varepsilon}(\cdot,y)|^q\,dx \lesssim R^{(2-d)q+d}.\]
This proves the estimate \ref{160802@eq1a}.

The proof of \eqref{160802@eq1b} is similar.
From the previous lemma we have for all $0 < \rho \le R_0$
\[\int_{\bR^d\setminus B_\rho(y)} |DV_{\varepsilon}(\cdot,y)|^2 \,dx \lesssim \rho^{2-d}.\]
Let $0<t<\infty$ and denote 
\[B(t)=\{x\in \bR^d:|DV_{\varepsilon}(x,y)|>t\}.\] 
By performing the same procedure, we can obtain \eqref{160802@eq1b}.
\end{proof}
\item
Similar to the previous lemmas, we prove uniform estimates for $\Pi_\varepsilon(\cdot,y)$.

\begin{lemma}		
\label{L54}
If Assumption \ref{ASSA} holds, then for any  $y\in \bR^d$, $0<R\le R_0$, and $\varepsilon>0$
\begin{equation}		\label{160802@eq1c}
\|\Pi_\varepsilon(\cdot,y)\|_{L_2(\bR^d\setminus B_R(y))} \lesssim_{d,\lambda,C_0,\alpha_0} R^{1-d/2}.
\end{equation}
Moreover, for any  $y\in \bR^d$, $0<R\le R_0$, and $\varepsilon>0$
\begin{equation}		\label{161122@eq4a}
\|\Pi_{\varepsilon}(\cdot,y)\|_{L_q(B_R(y))}\lesssim_{d,\lambda,C_0,\alpha_0,q} R^{1-d+d/q}, \quad q\in [1,d/(d-1)).
\end{equation}
\end{lemma}

\begin{proof}
If $\varepsilon\ge R/2$, then one can easily check \eqref{160802@eq1c} from \eqref{160727@eq1a}.
So, we assume $\varepsilon\in (0,R/2)$.
Let $\cD$ and $\eta$ be as in the proof of Lemma \ref{L52}.
Let $\varphi\in {Y}^1_2(\bR^d)^d$ be a solution to the divergence equation 
\[\divg \varphi=\Pi_{\varepsilon}^k(\cdot,y) \chi_{\bR^d\setminus \overline{B_R(y)}} \quad \text{in }\, \bR^d\]
satisfying
\begin{equation}		\label{161010@eq1a}
\|\varphi\|_{Y^1_2(\bR^d)}\lesssim_{d} \|\Pi_{\varepsilon}^k(\cdot,y)\|_{L_2(\bR^d\setminus \overline{B_{R}(y)})}.
\end{equation}
From the definition of the averaged fundamental solution $(V_\varepsilon(\cdot,y), \Pi_\varepsilon(\cdot,y))$ with a test function $(1-\eta)\varphi$, we obtain 
\begin{equation}
\label{E51}
\begin{split}
&\int_{\bR^d}A^{\alpha\beta}D_\beta V_\varepsilon^{\cdot k}(\cdot,y)\cdot D_\alpha((1-\eta)\varphi)\,dx
- \int_{\bR^d}\Pi_\varepsilon^k(\cdot,y)\divg ((1-\eta)\varphi)\,dx \\
&= \fint_{B_\varepsilon(y)}(1-\eta)\varphi^k\,dx = 0,
\end{split}
\end{equation}
where the last equality follows from the fact that the integrand vanishes in the domain of integration.
We notice that 
\begin{equation}
\label{E52}
\begin{split}
&\int_{\bR^d}\Pi_\varepsilon^k(\cdot,y)\divg ((1-\eta)\varphi)\,dx\\
&=\int_{\bR^d}\Pi_\varepsilon^k(\cdot,y)\divg \varphi\,dx-\int_{\bR^d}\Pi_\varepsilon^k(\cdot,y)\divg (\eta\varphi)\,dx\\
&=\int_{\bR^d\setminus \overline{B_R(y)}}|\Pi^k_\varepsilon(\cdot,y)|^2\,dx
-\int_{B_R(y)}\big(\Pi^k_\varepsilon(\cdot,y)-(\Pi^k_\varepsilon(\cdot,y))_{\cD}\big)D\eta\cdot  \varphi\,dx
\end{split}
\end{equation}
due to $\divg \varphi=0$ in $B_R(y)$.
Since 
\[\cL V_{\varepsilon}^{\cdot k}(\cdot,y) + \nabla \Pi^k_\varepsilon(\cdot,y)=0 \quad \text{in }\, \cD,\]
it follows from Lemma \ref{160808@lem2} that 
\begin{equation}		\label{161010@eq1}
\int_{\cD}|\Pi_\varepsilon^k(\cdot,y)-(\Pi_{\varepsilon}^k(\cdot,y))_{\cD}|^2\,dx 
\lesssim_{d,\lambda,C_0,\alpha_0} \int_{\cD}|DV_{\varepsilon}^{\cdot k}(\cdot,y)|^2\,dx.
\end{equation}
Using Young's inequality, H\"older's inequality, \eqref{161010@eq1a},  and \eqref{161010@eq1} we obtain that 
\begin{equation}
\label{E53}
\int_{B_R(y)}\big(\Pi^k_\varepsilon(\cdot,y)-(\Pi^k_\varepsilon(\cdot,y))_{\cD}\big)D\eta\cdot  \varphi\,dx
\lesssim_{d,\lambda,C_0,\alpha_0} \int_{\cD}|DV_{\varepsilon}^{\cdot k}(\cdot,y)|^2\,dx.
\end{equation}
Similarly, using Young's inequality, H\"older's inequality, and \eqref{161010@eq1a}, we obtain that for all positive number $\varepsilon$ 
\begin{equation}
\label{E54}
\begin{split}
&\int_{\bR^d}A^{\alpha\beta}D_\beta V_\varepsilon^{\cdot k}(\cdot,y)\cdot D_\alpha((1-\eta)\varphi)\,dx \\
&\lesssim \varepsilon \int_{\bR^d} |D((1-\eta)\varphi)|^2 \,dx 
+ C_\varepsilon \int_{\bR^d\setminus \overline{B_{R/2}(y)}} |D V_\varepsilon^{\cdot k}(\cdot,y)|^2 \,dx \\
&\lesssim \varepsilon \int_{\bR^d\setminus \overline{B_R(y)}}|\Pi^k_\varepsilon(\cdot,y)|^2\,dx
+ C_\varepsilon \int_{\bR^d\setminus \overline{B_{R/2}(y)}} |D V_\varepsilon^{\cdot k}(\cdot,y)|^2 \,dx.
\end{split}
\end{equation}
By choosing a small $\varepsilon$ and combining \eqref{E51}, \eqref{E52}, \eqref{E53}, and \eqref{E54}, we get 
\[\int_{\bR^d\setminus \overline{B_R(y)}}|\Pi^k_\varepsilon(\cdot,y)|^2\,dx 
\lesssim_{d,\lambda} \int_{\bR^d\setminus \overline{B_{R/2}(y)}}|DV_\varepsilon^{\cdot k}(\cdot,y)|^2\,dx.\]
Finally, we have from \eqref{160802@eq1} 
\[\int_{\cD}|DV_{\varepsilon}^{\cdot k}(\cdot,y)|^2\,dx
\lesssim_{d,\lambda,C_0,\alpha_0} R^{2-d},\]
so we get desired estimate \eqref{160802@eq1c}.

The proof of \eqref{161122@eq4a} is very similar but using \eqref{160802@eq1c} instead of \eqref{160802@eq1}.
\end{proof}

\item
Let  $y\in \bR^d$ and $q<d/(d-1)$.
By Lemma \ref{L52}, Lemma \ref{L54} and the weak compactness, there exists functions 
\begin{align*}
V_{\operatorname{ext}}\in Y^1_2(\bR^d\setminus \overline{B_{R_0/2}(y)})^{d\times d}, &\quad V_{\operatorname{int}}\in W^1_q(B_{R_0}(y))^{d\times d}, \\
\Pi_{\operatorname{ext}}\in L_2(\bR^d\setminus \overline{B_{R_0/2}(y)}), &\quad \Pi_{\operatorname{int}}\in L_q(B_{R_0}(y))
\end{align*}
and a sequence $\{\varepsilon_\rho\}_{\rho=1}^\infty$ tending to zero such that 
\begin{equation}
\label{160821@eq1a}
\begin{split}
&V_{\varepsilon_\rho}(\cdot,y)\rightharpoonup V_{\operatorname{ext}} \quad \text{weakly in }\, Y^1_2(\bR^d\setminus \overline{B_{R_0/2}(y)}),\\
&V_{\varepsilon_\rho}(\cdot,y)\rightharpoonup V_{\operatorname{int}}\quad \text{weakly in }\, W^1_q(B_{R_0}(y)),\\
\end{split}
\end{equation}
and 
\begin{equation}
\label{160821@eq2a}
\begin{split}
&\Pi_{\varepsilon_\rho}(\cdot,y)\to \Pi_{\operatorname{ext}} \quad \text{in }\, L_2(\bR^d\setminus \overline{B_{R_0/2}(y)}),\\
&\Pi_{\varepsilon_\rho}(\cdot,y)\rightharpoonup \Pi_{\operatorname{int}}\quad \text{weakly in }\, L_q(B_{R_0}(y)).
\end{split}
\end{equation}
Oberve that $V_{\operatorname{ext}}=V_{\operatorname{int}}$ on $B_{R_0}(y)\setminus \overline{B_{R_0/2}(y)}$, and we define 
\[V(\cdot,y):=\left\{
\begin{aligned}
V_{\operatorname{ext}} &\quad \text{on }\, \bR^d\setminus \overline{B_{R_0}(y)},\\
V_{\operatorname{ext}}=V_{\operatorname{int}} &\quad \text{on }\, B_{R_0}(y)\setminus \overline{B_{R_0/2}(y)},\\
V_{\operatorname{int}} &\quad \text{on }\, B_{R_0/2}(y),
\end{aligned}
\right.\]
and similarly 
\[\Pi(\cdot,y):=\left\{
\begin{aligned}
\Pi_{\operatorname{ext}} &\quad \text{on }\, \bR^d\setminus \overline{B_{R_0}(y)},\\
\Pi_{\operatorname{ext}}=\Pi_{\operatorname{int}} &\quad \text{on }\, B_{R_0}(y)\setminus \overline{B_{R_0/2}(y)},\\
\Pi_{\operatorname{int}} &\quad \text{on }\, B_{R_0/2}(y).
\end{aligned}
\right.\]
By \eqref{160802@eq1}, \eqref{160802@eq1c}, and a diagonalization process, there exists a subsequence, still denoted by $\{\varepsilon_\rho\}_{\rho=1}^\infty$, such that 
\begin{equation}		\label{160821@eq4}
V_{\varepsilon_\rho}(\cdot,y)\rightharpoonup V(\cdot,y) \quad \text{weakly in }\, Y^1_2(\bR^d\setminus \overline{B_{r}(y)}), \quad \forall r>0,
\end{equation}
and 
\begin{equation}		\label{160821@eq4a}
\Pi_{\varepsilon_\rho}(\cdot,y)\to \Pi(\cdot,y) \quad \text{in }\, L_2(\bR^d\setminus \overline{B_{r}(y)}), \quad \forall r>0.
\end{equation}
\item
We shall show $(V,\Pi)$ satisfies the conditions in Definition \ref{D23}.
Obviously, it satisifes the condition $(a)$.\\
{\emph{Verifying $(b)$.}}
Let $y\in \bR^d$.
Since $\divg V_{\varepsilon_\rho}^{\cdot k}(\cdot,y)=0$ in $\bR^d$, by using \eqref{160821@eq1a} and \eqref{160821@eq4}, one can easily check that \eqref{160822@eq1} holds.  
To show \eqref{160821@eq5}, 
we notice from \eqref{160727@eq1} that 
\begin{align}
\nonumber
\phi^k(y)&=\lim_{\rho\to \infty}\fint_{B_{\varepsilon_\rho}(y)}\phi^k(x)\,dx\\
\label{160821@eq6}
&=\lim_{\rho\to \infty}\left(\int_{\bR^d}A^{\alpha\beta}D_\beta V^{\cdot k}_{\varepsilon_\rho}(\cdot,y)\cdot D_\alpha \phi\,dx 
- \int_{\bR^d}\Pi^k_{\varepsilon_\rho}(\cdot,y)\divg \phi\,dx\right)
\end{align}
for any $\phi\in C^\infty_0(\bR^d)^d$.
Using  \eqref{160821@eq1a} and \eqref{160821@eq4}, we have  
\begin{align}
\nonumber
&\lim_{\rho\to \infty}\int_{\bR^d}A^{\alpha\beta}D_\beta V^{\cdot k}_{\varepsilon_\rho}(\cdot,y)\cdot D_\alpha \phi\,dx\\
\nonumber
&=\lim_{\rho\to \infty}\left(\int_{B_{R_0}(y)}A^{\alpha\beta}D_\beta V^{\cdot k}_{\varepsilon_\rho}(\cdot,y)\cdot D_\alpha \phi\,dx+\int_{\bR^d\setminus B_{R_0}(y)}A^{\alpha\beta}D_\beta V^{\cdot k}_{\varepsilon_\rho}(\cdot,y)\cdot D_\alpha \phi\,dx\right)\\
\nonumber
&=\int_{B_{R_0}(y)}A^{\alpha\beta}D_\beta V^{\cdot k}(\cdot,y)\cdot D_\alpha \phi\,dx+\int_{\bR^d\setminus B_{R_0}(y)}A^{\alpha\beta}D_\beta V^{\cdot k}(\cdot,y)\cdot D_\alpha \phi\,dx\\
\label{160821@eq6a}
&=\int_{\bR^d}A^{\alpha\beta}D_\beta V^{\cdot k}(\cdot,y)\cdot D_\alpha \phi\,dx.
\end{align}
Similarly, we obtain by  \eqref{160821@eq2a} and \eqref{160821@eq4a} that 
\[\lim_{\rho\to \infty}\int_{\bR^d}\Pi^k_{\varepsilon_\rho}(\cdot,y)\divg \phi\,dx=\int_{\bR^d}\Pi^k(\cdot,y)\divg \phi\,dx.\]
From this together with \eqref{160821@eq6} and \eqref{160821@eq6a}, we get \eqref{160821@eq5}.\\
{\emph{Verifying $(c)$.}}
It suffices to prove that  \eqref{170307@eq1} holds under the assumptions \eqref{170307@eq2} and \eqref{170307@eq3}. 
Let $q_0>d$.
By the uniform estimates \eqref{160802@eq1a}, \eqref{160802@eq1b} and \eqref{161122@eq4a}, we may assume that 
\begin{equation}		\label{160827@eq1a}
\begin{aligned}
V_{\varepsilon_\rho}(\cdot,y)\rightharpoonup V(\cdot,y) &\quad \text{weakly in }\, L_{q_0/(q_0-2)}(B_{R_0}(y)),\\
DV_{\varepsilon_\rho}(\cdot,y)\rightharpoonup DV(\cdot,y) &\quad \text{weakly in }\, L_{q_0/(q_0-1)}(B_{R_0}(y)),\\
\Pi_{\varepsilon_\rho}(\cdot,y)\rightharpoonup \Pi(\cdot,y) &\quad \text{weakly in }\, L_{q_0/(q_0-1)}(B_{R_0}(y)).
\end{aligned}
\end{equation}
Let $(u,p)\in Y^1_2(\bR^d)\times L_2(\bR^d)$ be the weak solution of  \eqref{170307@eq3}.
Then by testing with $V_{\varepsilon_\rho}^{\cdot k}(\cdot,y)$ to \eqref{170307@eq3} and setting $\phi=u$ in \eqref{160727@eq1}, we have (see e.g., \eqref{160822@eq3})
\[\fint_{B_{\varepsilon_\rho}(y)}u^k\,dx=\int_{\bR^d}V^{\cdot k}_{\varepsilon_\rho}(\cdot,y)\cdot  f\,dx
-\int_{\bR^d}D_\alpha V^{\cdot k}_{\varepsilon_\rho}(\cdot,y)\cdot f_\alpha\,dx - \int_{\bR^d}\Pi^k_{\varepsilon_\rho}(\cdot,y)g\,dx.\]
Then similar to the proof of $(b)$,
by using \eqref{160821@eq4}, \eqref{160821@eq4a}, and \eqref{160827@eq1a}, we  conclude that 
\[u^k(y)=\int_{\bR^d}V^{\cdot k}(\cdot,y)\cdot  f\,dx-\int_{\bR^d}D_\alpha V^{\cdot k}(\cdot,y)\cdot f_\alpha\,dx - \int_{\bR^d}\Pi^k(\cdot,y)g\,dx,\]
which implies the identity  \eqref{170307@eq1}.\\
\item
Let us fix $y\in \bR^d$ and $R\in (0,R_0]$.
By \eqref{160802@eq1} and \eqref{160821@eq4}, we obtain for $\phi\in C^\infty_0(\bR^d)^d$ 
\begin{align*}
\left|\int_{\bR^d\setminus B_R(y)}V^{\cdot k}(\cdot,y)\cdot  \phi\,dx  \right|&=\lim_{\rho\to \infty}\left|\int_{\bR^d\setminus B_R(y)}V^{\cdot k}_{\varepsilon_\rho}(\cdot,y)\cdot \phi\,dx\right|\\
&\lesssim_{d,\lambda,C_0,\alpha_0} R^{1-d/2}\|\phi\|_{L_{2d/(d+2)}(\bR^d\setminus B_R(y))},
\end{align*}
which implies 
\[\|V^{\cdot k}(\cdot,y)\|_{L_{2d/(d-2)}(\bR^d\setminus B_R(y))} \lesssim_{d,\lambda,C_0,\alpha_0} R^{1-d/2}.\]
Using this argument together with Lemmas \ref{L52} and \ref{L54}, it is routine to check the estimates $i) - v)$ in Corollary \ref{MRC}.

To get the pointwise estimate \eqref{160822_eq1}, let $x,\,y\in \bR^d$, and $0<R:=|x-y|\le R_0$.
By the condition $(b)$ in the definition, we find that $(V^{\cdot k}(\cdot,y),\Pi^k(\cdot,y))$ satisfies 
\[
\left\{
\begin{aligned}
\divg V^{\cdot k}(\cdot,y)=0 &\quad \text{in }\, B_{R/2}(x),\\
\cL V^{\cdot k}(\cdot,y) + \nabla\Pi^k(\cdot,y)=0 &\quad \text{in }\, B_{R/2}(x).
\end{aligned}
\right.
\]
Therefore, by Lemma \ref{161121@lem7} and $i)$ in Corollary \ref{MRC},  we conclude that 
\begin{align*}
|V^{\cdot k}(x,y)|
&\lesssim R^{-d}\|V^{\cdot k}(\cdot,y)\|_{L_1(B_{R/2}(x))}\\
&\lesssim R^{1-d/2}\|V^{\cdot k}(\cdot,y)\|_{L_{2d/(d-2)}(\bR^d\setminus B_{R/2}(y))} \lesssim R^{2-d},
\end{align*}
which implies the pointwise estimate \eqref{160822_eq1}.
\item
Finally, we prove the uniqueness of the fundamental solution $(V,\Pi)$.
Let $(\tilde{V},\tilde{\Pi})$ be another pair satisfying the condition $(c)$ in Definition \ref{D23}.
By the unique solvability of Stokes system 
\[\int_{\bR^d}V(\cdot,y)^{\operatorname{tr}}f\,dx - \int_{\bR^d}\Pi(\cdot,y)g\,dx=\int_{\bR^d}\tilde{V}(\cdot,y)^{\operatorname{tr}}f\,dx - \int_{\bR^d}\tilde{\Pi}(\cdot,y)g\,dx\]
for all $f\in C^\infty_0(\bR^d)^d$ and $g\in C^\infty_0(\bR^d)$.
Thus, we should have for almost all $x,\,y\in \bR^d$ 
\[(V(x,y),\Pi(x,y))=(\tilde{V}(x,y),\tilde{\Pi}(x,y)).\]
\end{enumerate}

This completes the proof of Theorem \ref{MRA}.
\\
We end this section by giving the proof of  Corollary \ref{161122@cor1}, which is a slight modification of that of \cite[Eq. (2.5)]{arXiv:1503.07290v3}.

Let  $({}^*V, {}^*\Pi)$ and $({}^*V_{\delta}, {}^*\Pi_{\delta})$ be the  fundamental solution and the averaged fundamental solution for $\cL^*$, respectively; i.e., for $y\in \bR^d$ and $k\in \{1,\ldots,d\}$, the pair $({}^*V_{\delta}^{\cdot k}(\cdot,y), {}^*\Pi_\delta^{k}(\cdot,y))$, where ${}^*V_{\delta}^{\cdot k}$ is the $k$-th column of ${}^*V_{\delta}$, is the weak solution in ${Y}^1_2(\bR^d)^d\times L_2(\bR^d)$ of 
\begin{equation}		\label{160828@eq10}
\left\{
\begin{aligned}
\cL^* ({}^*V_{\delta}^{\cdot k}(\cdot,y)) + \nabla ({}^*\Pi^k_{\delta}(\cdot,y))=\frac{I_{B_{\delta}(y)}}{|B_\delta(y)|}e_k \quad \text{in }\, \bR^d,\\
\divg ({}^*V_{\delta}^{\cdot k}(\cdot ,y))= 0 \quad \text{in }\, \bR^d.
\end{aligned}
\right.
\end{equation}
Then $({}^*V, {}^*\Pi)$ and $({}^*V_{\delta}, {}^*\Pi_{\delta})$ satisfy counterparts of results in Theorem \ref{MRA}.

\begin{lemma}		
\label{L55}
Let $y\in \bR^d$.
For any compact set $K\subset \bR^d\setminus \{y\}$, there exist sequences $\{\varepsilon_\rho\}^\infty_{\rho=1}$ and $\{\delta_\tau\}_{\tau=1}^\infty$ tending to zero such that 
\begin{align*}
V_{\varepsilon_\rho}(\cdot,y)\to V(\cdot,y) \quad \text{uniformly on }\, K,\\
{}^*V_{\delta_\tau}(\cdot,y)\to {}^*V(\cdot,y) \quad \text{uniformly on }\, K.
\end{align*}
\end{lemma}

\begin{proof}
The proof is the same as that of \cite[Lemma 4.4]{arXiv:1503.07290v3}.
\end{proof}

Now we are ready to prove the case $\Omega=\bR^d$ in Corollary \ref{161122@cor1}.
Let $x,\, y\in \bR^d$, $x\neq y$, and $k,\ell=1,\ldots,d$.
By setting $\phi={}^*V_\delta^{\cdot \ell}(\cdot,x)$ in \eqref{160727@eq1} and by using $V_{\varepsilon}^{\cdot k}(\cdot,y)$ as a test function to \eqref{160828@eq10}, we get 
\begin{equation}		\label{160906@eq1}
\Gamma_{\varepsilon,\delta}^{k\ell}:=\fint_{B_{\varepsilon}(y)}{}^*V^{k\ell}_{\delta}(\cdot,x)\,dz=\fint_{B_{\delta}(x)}V^{\ell k}_{\varepsilon}(\cdot,y)\,dz.
\end{equation}
Let $\{\varepsilon_\rho\}$ and $\{\delta_\tau\}$ be sequences in Lemma \ref{L55}.
Then by the continuity of $V_{\varepsilon_\rho}(\cdot,y)$ and Lemma \ref{L55}, we have 
\[\lim_{\rho\to \infty}\lim_{\tau\to \infty}\Gamma^{k\ell}_{\varepsilon_\rho,\delta_\tau}=\lim_{\rho\to \infty}\lim_{\tau\to \infty}\fint_{B_{\delta_\tau}(x)}V^{\ell k}_{\varepsilon_\rho}(\cdot,y)\,dz=V^{\ell k}(x,y).\]
Similarly, by the continuity of ${}^*V(\cdot,y)$ and Lemma \ref{L55}, we obtain 
\[\lim_{\rho\to \infty}\lim_{\tau\to \infty}\Gamma^{k\ell}_{\varepsilon_\rho,\delta_\tau}=\lim_{\rho\to \infty}\lim_{\tau\to \infty}\fint_{B_{\varepsilon_\rho}(y)}{}^*V^{k\ell}_{\delta_\tau}(\cdot,x)\,dz={}^*V^{k\ell}(y,x).\]
We thus have 
\begin{equation}		\label{160906@eq2}
V^{\ell k}(x,y)={}^*V^{k\ell}(y,x),
\end{equation}
which gives the identity \eqref{161120@eq1}.
We notice from \eqref{160906@eq1} and \eqref{160906@eq2} that 
\begin{align*}
V^{k\ell}_{\varepsilon}(x,y)&=\lim_{\tau\to 0}\fint_{B_{\delta_\tau}(x)}V_\varepsilon^{k\ell}(\cdot,y)\,dz=\lim_{\tau\to 0}\fint_{B_{\varepsilon}(y)}{}^*V^{k\ell}_{\delta_\tau}(\cdot,x)\,dz\\
&=\fint_{B_{\varepsilon}(y)}{}^*V^{k\ell}(\cdot,x)\,dz=\fint_{B_{\varepsilon}(y)}V^{\ell k}(x,\cdot)\,dz.
\end{align*}
This justifies why we call it the averaged fundamental solution.
Finally, the representation formula \eqref{161122@eq8} is an easy consequence of the identity \eqref{160906@eq2} and the counterpart of  \eqref{170307@eq1}.

This completes the proof of the case $\Omega=\bR^d$ in Corollary \ref{161122@cor1}.
The case of $\Omega=\bR^d_+$ can be treated in a similar way.

\section{Proof of Theorem \ref{MRB}}
\label{S6}

The proof is a slight modification of the proof of Theorem \ref{MRA}.
For each $\varepsilon>0$, $y\in \bR_+^d$, and $k\in \{1,\ldots,d\}$ we denote 
\[f_{\varepsilon;y,k} 
= \frac{\chi_{\bR^d_+\cap B_\varepsilon(y)}}{|\bR^d_+\cap B_\varepsilon(y)|}e_k\]
where $\chi_E$ is the characteristic function and $e_k$ is the $k$-th unit vector in $\bR^d$.
We define \emph{an averaged Green function} $(V^{\cdot k}_\varepsilon(\cdot,y), \Pi_\varepsilon^k(\cdot,y)) \in \rY^1_2(\bR^d_+)^d\times L_2(\bR^d_+)$ as the unique weak solution to the problem
\[\left\{
\begin{aligned}
\cL u+\nabla p= f_{\varepsilon;y,k} &\quad \text{in }\, \bR^d_+,\\
\divg u=0 &\quad \text{in }\, \bR^d_+.
\end{aligned}
\right.\]
Using \eqref{160721@eq2} we have 
\begin{equation}		\label{160908@eq1}
\|DV_\varepsilon(\cdot,y)\|_{L_2(\bR^d_+)}+\|\Pi_\varepsilon(\cdot,y)\|_{L_2(\bR^d_+)}\lesssim_{d,\lambda} \varepsilon^{1-d/2}, \quad \forall \varepsilon>0.
\end{equation}
Moreover, for all $x,\,y\in \bR^d_+$ and $\varepsilon>0$ satisfying 
\[0<\varepsilon\le \frac{|x-y|}{3}\le \frac{1}{2}\min\{ d_x, d_y, R_0\},\]
we obtain the pointwise estimate
\begin{equation}		\label{160908@eq1a}
|V_\varepsilon(x,y)|\lesssim_{d,\lambda,C_0,\alpha_0} |x-y|^{2-d}
\end{equation}
by repeating the same argument as in the proof of Lemma \ref{L51}.
The pointwise estimate \eqref{160908@eq1a} can also yield the following uniform estimates.

\begin{lemma}		\label{160920@lem1}
For any $y\in \bR^d_+$, $0<R\le \min\{d_y,R_0\}$, and $\varepsilon>0$, \begin{align}		
\label{160927@eq1b}
\|V_\varepsilon(\cdot,y)\|_{Y^1_2(\bR^d_+\setminus \overline{B_R(y)})}
&\lesssim_{d,\lambda, C_0,\alpha_0} R^{1-d/2},\\
\label{161018@EE2}
\|\Pi_\varepsilon(\cdot,y)\|_{L_2(\bR^d_+\setminus \overline{B_R(y)})}
&\lesssim_{d,\lambda, C_0,\alpha_0} R^{1-d/2}.
\end{align}
Moreover, for any $y\in \bR^d_+$, $0<R\le \min\{d_y,R_0\}$, and $\varepsilon>0$, 
\begin{align}
\label{160927@eq1}
\|V_\varepsilon(\cdot,y)\|_{L_q(\bR^d_+\cap B_R(y))} 
&\lesssim_{d,\lambda,C_0,\alpha_0,q} R^{2-d+d/q}, \quad q\in [1,d/(d-2)),\\
\label{160927@eq1a}
\|DV_\varepsilon(\cdot,y)\|_{L_q(\bR^d_+\cap B_R(y))} 
&\lesssim_{d,\lambda,C_0,\alpha_0,q} R^{1-d+d/q}, \quad q\in [1,d/(d-1)),\\
\label{161018@EE1}
\|\Pi_\varepsilon(\cdot,y)\|_{L_q(\bR^d_+\cap B_R(y))} 
&\lesssim_{d,\lambda,C_0,\alpha_0,q} R^{1-d+d/q}, \quad q\in [1,d/(d-1)).
\end{align}
\end{lemma}

\begin{proof}
Let $R_0':=\min\{d_y,R_0\}$.
The proof of \eqref{160927@eq1b} for $R\le R_0'/2$ is the same as the proof of \eqref{160802@eq1} by using \eqref{160908@eq1} and \eqref{160908@eq1a}.
Since $R_0'/2$ and $R_0'$ are comparable to each other, it is not hard to see that \eqref{160927@eq1b} holds for $R\in (R_0'/2, R_0']$.
Therefore, we have \eqref{160927@eq1b}.
To show \eqref{161018@EE2}, we notice from Lemma \ref{160810@lem1} that 
there exists $\varphi\in \rY^1_2(\bR^d_+)^d$ such that 
\[\divg \varphi=\Pi_{\varepsilon}^k(\cdot,y)I_{\bR^d_+\setminus \overline{B_R(y)}} \quad \text{in }\, \bR^d_+\]
satisfying 
\[\|\varphi\|_{Y^1_2(\bR^d_+)}\le N(d)\|\Pi_{\varepsilon}^k(\cdot,y)\|_{L_2(\bR^d_+\setminus \overline{B_{R}(y)})}.\]
Using this and \eqref{160927@eq1b}, one can easily obtain \eqref{161018@EE2} just following the proof of \eqref{160802@eq1c}. 
The estimates \eqref{160927@eq1} -- \eqref{161018@EE1} are deduced from \eqref{160927@eq1b} and \eqref{161018@EE2} in the same way as \eqref{160802@eq1a}, \eqref{160802@eq1b}, and \eqref{161122@eq4a} are deduced from \eqref{160802@eq1} and \eqref{160802@eq1c}.
We omit the details.
\end{proof}

The proof of Theorem \ref{MRB} is based on Lemma \ref{160920@lem1} and exactly the same argument in the proof of Theorem \ref{MRA}.
We can find the Green function  $(V,\Pi)$ satisfying the pointewise estimate in Theorem \ref{MRB} and all the estimates for $\Omega=\bR_+^d$ in Corollary \ref{MRC}.
We omit the repeated details. 

\section{Proof of Theorem \ref{MRD}}
\label{S7}
Suppose $A^{\alpha\beta}_0=A^{\alpha\beta}_0(x_1)$ satisfy \eqref{161123@EQ1} and denote 
\begin{equation}		\label{161125@eq1}
\cL_0 u=-D_\alpha(A^{\alpha\beta}_0D_\beta u).
\end{equation}

In the lemma below, we provide  interior $L_\infty$-estimates for $Du$ and $p$, where  $(u,p)$ is a solution of  
\begin{equation}		\label{161123@EQ2}
\cL_0u + \nabla p=0, \quad \divg u=0.
\end{equation}
The results in the following lemma were proved by Dong--Kim \cite[Section 4]{arXiv:1604.02690v2}. 
Actually, they proved $L_\infty$-estimates of $D_{x'}u$ and certain linear combinations of $Du$ and $p$. 
Using this and the argument in \cite[Section 6]{arXiv:1604.02690v2}, one can easily show $L_\infty$-estimates for $Du$ and $p$.
Here, we reproduce it for the reader's convenience by rearranging the proof in \cite{arXiv:1604.02690v2}.

\begin{lemma}		\label{161014@lem5}
If $(u,p)\in W^1_2(B_2)^d\times L_2(B_2)$ satisfies \eqref{161123@EQ2} in $B_2$, then 
\begin{equation}		\label{161014@eq6}
\|Du\|_{L_\infty(B_1)}\lesssim_{d,\lambda} \|Du\|_{L_2(B_2)}
\end{equation}
and 
\begin{equation}		\label{161014@eq6a}
\|p\|_{L_\infty(B_1)}\lesssim_{d,\lambda} \|Du\|_{L_2(B_2)} + \|p\|_{L_2(B_2)}.
\end{equation}
\end{lemma}

\begin{proof}
From \cite[Lemma 4.3]{arXiv:1604.02690v2}, we have 
\begin{equation}		\label{161124@eq1}
\|D_{x'}u\|_{L_\infty(B_1)}+\sum_{i=2}^d\|U^i\|_{L_\infty(B_1)}\lesssim \|Du\|_{L_2(B_2)}
\end{equation}
and
\begin{equation}		\label{161124@eq2}
\|U^1\|_{L_\infty(B_1)}\lesssim \|Du\|_{L_2(B_2)} + \|p\|_{L_2(B_2)},
\end{equation}
where 
\[U^1=(A^{1\beta}_0)_{1j}D_\beta u^j+p e_1, \quad U^i=(A^{1\beta}_0)_{ij}D_\beta u^j, \quad i=2,\ldots,d.\]
Since $\divg u=0$, we obtain from \eqref{161124@eq1} that 
\begin{equation}		\label{161014@eq6b}
\|D_1u^1\|_{L_\infty(B_1)}\lesssim \|Du\|_{L_2(B_2)}.
\end{equation}
Since
\[\sum_{j=2}^d(A^{11}_0)_{ij}D_1u^j=U^i-\sum_{j=1}^d\sum_{\beta=2}^d(A^{1\beta}_0)_{ij}D_\beta u^j-(A^{11}_0)_{i1}D_1u^1, \quad i\in\{2,\ldots,d\},\]
we multiply both sides by $D_1u^i$ and then sum over $i=2,\ldots,d$ to obtain 
\[\sum_{i,j=2}^dA^{11}_{ij}D_1u^j D_1u^i=\sum_{i=2}^d U^i D_1u^i-\sum_{j=1}^d\sum_{i,\beta=2}^d(A^{1\beta}_0)_{ij}D_\beta u^j D_1u^i-\sum_{i=2}^d (A^{11}_0)_{i1}D_1u^1D_1u^i.\]
Thus, by the ellipticity condition \eqref{161123@EQ1} and Young's inequality, we have 
\[\sum_{j=2}|D_1u^j(x)|^2\lesssim_{d,\lambda} \sum_{i=2}^d |U^i(x)|^2 + |D_{x'}u(x)|^2 + |D_1u^1(x)|^2\]
for almost all $x\in B_1$.
Taking the norm $\|\cdot \|_{L_\infty(B_1)}$ to both sides of the above inequality, and then using \eqref{161124@eq1} and  \eqref{161014@eq6b}, we get \eqref{161014@eq6}.
Finally, since 
\[p e_1=U^1-(A^{1\beta}_0)_{1j}D_\beta u^j,\]
we get \eqref{161014@eq6a} from  \eqref{161014@eq6} and \eqref{161124@eq2}.
\end{proof}

\begin{corollary}		\label{161014@cor5}
Let $0<r<R$.
If $(u,p)\in W^1_2(B_R)^d\times L_2(B_R)$ satisfies \eqref{161123@EQ2} in $B_R$,
then 
\[\|Du\|_{L_\infty(B_r)}\lesssim_{d,\lambda} (R-r)^{-d/2}\|Du\|_{L_2(B_R)}\]
and 
\[\|p\|_{L_\infty(B_r)}\lesssim_{d,\lambda} (R-r)^{-d/2}\big(\|Du\|_{L_2(B_R)} + \|p\|_{L_2(B_R)}\big).\]
\end{corollary}

\begin{proof}
Based on Lemma \ref{161014@lem5} with scaling and a well known argument in \cite[p. 80]{MR1239172}, one can easily obtain the desired estimates.
We omit the details.
\end{proof}

Now we are ready to prove Theorem \ref{MRD}.
We only prove the case (b) because (a) is its special case.

\begin{enumerate}[\bf{Step} 1)]
\item 
Set
\begin{equation}		\label{161125@dd1}
\omega(R):=\sup_{x\in \bR^d}\sup_{r\le R}\fint_{B_r(x)}\bigg|A^{\alpha\beta}(y_1,y')-\fint_{B_r'(x')}A^{\alpha\beta}(y_1,z')\,dz'\bigg|\,dy,
\end{equation}
where $A^{\alpha\beta}$ are coefficients of $\cL$.
Assume 
\[\omega(R_0)\le \gamma < 1\]
where $\gamma$ is a positive constant to be chosen later.
Let $(u,p)\in W^1_2(B_R(x^0))^d\times L_2(B_R(x^0))$ satisfy for $0<R\le R_0$
\[\left\{
\begin{aligned}
\cL u+\nabla p=0 &\quad \text{in }\, B_R(x^0),\\
\divg u=0 &\quad \text{in }\, B_R(x^0).
\end{aligned}
\right.\]
\item
Let $y=(y_1,y')$ and $B_{2r}(y)\subseteq B_R(x^0)$.
We denote
\[\cL_0u=-D_\alpha(A^{\alpha\beta}_0D_\beta u),\]
where 
\[A^{\alpha\beta}_0=A^{\alpha\beta}_0(x_1)=\fint_{B_r'(y')}A^{\alpha\beta}(x_1,z')\,dz'.\]
By the solvability of the Stokes system with the Dirichlet boundary condition (see, for instance, \cite[Lemme 3.1]{arXiv:1503.07290v3}), there exists a unique pair $(u_1,p_1)\in \mathring{W}^1_2(B_r(y))^d\times L_2(B_r(y))$ satisfying $\int_{B_r(y)}p_1\,dx=0$ and 
\[\left\{
\begin{aligned}
\cL_0 u_1 + \nabla p_1=-\cL u+\cL_0 u &\quad \text{in }\, B_r(y),\\
\divg u_1=0 &\quad \text{in }\, B_r(y).
\end{aligned}
\right.\]
Moreover, we have the following $L_2$-estimate:
\begin{equation}		\label{161015@eq1}
\|Du_1\|_{L_2(B_r(y))} \lesssim_{d,\lambda} \big\|(A^{\alpha\beta}-A_0^{\alpha\beta})D_\beta u\big\|_{L_2(B_r(y))}.
\end{equation}
By the reverse H\"older inequality (see  Lemma \ref{161127@lem1}), there exists a constant $q_0=q_0(d,\lambda)>2$ such that 
\begin{equation}		\label{161015@eq1a}
\left(\fint_{B_r(y)}|Du|^{q_0}\,dx\right)^{1/q_0}\lesssim_{d,\lambda} \left(\fint_{B_{2r}(y)}|Du|^2\,dx\right)^{1/2}.
\end{equation}
Applying H\"older's inequality and \eqref{161015@eq1a} to \eqref{161015@eq1}, we have 
\begin{equation}
\label{161015@eq2}
\begin{split}
\|Du_1\|_{L_2(B_r(y))} 
&\lesssim_{d,\lambda} r^{\frac{d}{2}}\left(\fint_{B_r(y)}|A^{\alpha\beta}-A^{\alpha\beta}_0|^{\frac{2q_0}{q_0-2}}\,dx\right)^{\frac{q_0-2}{2q_0}}\left(\fint_{B_r(y)}|Du|^{q_0}\,dx\right)^{\frac{1}{q_0}}\\
&\lesssim_{d,\lambda} r^{\frac{d}{2}} \gamma \left(\fint_{B_{2r}(y)}|Du|^{2}\,dx\right)^{\frac{1}{2}}\\
&\lesssim_{d,\lambda} \gamma \|Du\|_{L_2(B_{2r}(y))}.
\end{split}
\end{equation}
\item
Since $(u_2,p_2):=(u-u_1,p-p_1)$ satisfies
\[\left\{
\begin{aligned}
\cL_0 u_2 + \nabla p_2=0 &\quad \text{in }\, B_r(y),\\
\divg u_2=0 &\quad \text{in }\, B_r(y),
\end{aligned}
\right.\]
Corollary \ref{161014@cor5} implies that for $0<\rho<r$ 
\[\|Du_2\|_{L_2(B_\rho(y))}\lesssim \left(\frac{\rho}{r}\right)^{d/2}\|Du_2\|_{L_2(B_r(y))}.\]
Thus, from \eqref{161015@eq2}, we get  
\begin{equation}
\label{161015@eq2a}
\begin{split}
\|Du\|_{L_2(B_\rho(y))} 
&\le \|Du_1\|_{L_2(B_\rho(y))}+\|Du_2\|_{L_2(B_\rho(y))}\\
&\lesssim_{d,\lambda} \left(\left(\frac{\rho}{r}\right)^{d/2}+\gamma\right)\|Du\|_{L_2(B_{2r}(y))}.
\end{split}
\end{equation}
We note that it is trivially hold for $\rho\in [r,2r]$ and $B_{2r}(y)\subseteq B_R(x^0)$.
Let $B_r(y)\subseteq B_R(x^0)$ and $\alpha_0\in (0,1)$.
We can take $\gamma = \tau^{d/2}$ and choose a sufficiently small $\tau(d, \lambda, \alpha_0)\in (0,1)$ so that 
\[\|Du\|_{L_2(B_{\tau r}(y))} \le \tau^{\frac{d}{2}-1+\alpha_0}\|Du\|_{L_2(B_r(y))}.\]
Hence, by an iteration, we obtain that for $0<\rho<r$ 
\begin{equation}		
\label{161015@eq5}
\|Du\|_{L_2(B_{\rho}(y))} \lesssim_{d,\lambda,\alpha_0} \left(\frac{\rho}{r}\right)^{\frac{d}{2}-1+\alpha_0}\|Du\|_{L_2(B_r(y))}.
\end{equation}
\item
From \eqref{161015@eq5} we have for $y\in B_{R/4}(x^0)$ and $\rho\in (0, R/4)$ 
\begin{align*}
\|Du\|_{L_2(B_{\rho}(y))} 
&\lesssim_{d,\lambda,\alpha_0} \left(\frac{\rho}{R}\right)^{\frac{d}{2}-1+\alpha_0}\|Du\|_{L_2(B_{R/4}(y))}\\
&\lesssim_{d,\lambda,\alpha_0} \left(\frac{\rho}{R}\right)^{\frac{d}{2}-1+\alpha_0}\|Du\|_{L_2(B_{R/2}(x^0))}.
\end{align*}
From \eqref{160920@lem2}, we get 
\[\|Du\|_{L_2(B_{\rho}(y))} \lesssim_{d,\lambda,\alpha_0} \frac{\rho^{\frac{d}{2}-1+\alpha_0}}{R^{\frac{d}{2}+\alpha_0}}\|u\|_{L_2(B_R(x^0))}.\]
Therefore, the Morrey-Campanato theorem yields 
\begin{equation*}		
[u]_{C^{\alpha_0}(B_{R/4}(x^0))}\le C R^{-\frac{d}{2}-\alpha_0}\| u\|_{L_2(B_R(x^0))}.
\end{equation*}  
Finally, a standard covering argument yields
\[[u]_{C^{\alpha_0}(B_{R/2}(x^0))}\lesssim_{d,\lambda,\alpha_0} R^{-\alpha_0}\left(\fint_{B_R(x^0)}| u|^2\,dx\right)^{1/2}, \quad 0<R\le R_0.\]
\end{enumerate}
This completes the proof of Theorem \ref{MRD}.

\section{Proof of Theorem \ref{MRE}}
\label{S8}

The proof of the estimate \eqref{161019@eq5} is a modification of the argument for elliptic systems found in Kang--Kim \cite[Theorem 3.3]{ MR2718661}.
We divide the proof into several steps.

\begin{enumerate}[\bf{Step} 1)]
\item
Let  $x,\,y\in \bR^d_+$ and $0<R:=|x-y|\le \min\{R_0,R_1\}$.
We note that 
$(V^{\cdot k}(\cdot,y), \Pi^k(\cdot,y))$ satisfies 
\[\left\{
\begin{aligned}
\cL V^{\cdot k}(\cdot,y) + \nabla \Pi^k(\cdot,y)= 0&\quad \text{in }\, \bR^d_+\cap B_{R/2}(x),\\
\divg V^{\cdot k}(\cdot,y)=0 &\quad \text{in }\, \bR^d_+\cap B_{R/2}(x),\\
V^{\cdot k}(\cdot,y)=0 &\quad \text{on }\, \partial \bR^d_+.
\end{aligned}
\right.\]
If $d_x> R/8$, then since $B_{R/8}(x)\subset \bR^d_+$, by Lemma \ref{161121@lem7} $(a)$, we have
\begin{align}		
\nonumber
|V^{\cdot k}(x,y)| 
&\lesssim_{d,C_0, \alpha_0} R^{-d}\|V^{\cdot k}(\cdot,y)\|_{L_1(B_{R/8}(x))}\\
\label{161025@eq7}
&\lesssim_{d,C_0, \alpha_0} R^{-d}\|V^{\cdot k}(\cdot,y)\|_{L_1(\bR^d_+\cap B_{R/2}(x))}.
\end{align}
If $d_x\le R/8$, then we take $x^0\in \partial\bR^d_+$ satisfying $\operatorname{dist}(x,\partial \bR^d_+)=|x-x^0|$ so that 
\[x\in B_{3R/16}^+(x^0)\subset B_{3R/8}^+(x^0)\subset (\bR^d_+\cap B_{R/2}(x)).\]
By Lemma \ref{161121@lem7} $(b)$ 
\begin{align}		
\nonumber
|V^{\cdot k}(x,y)|&\lesssim_{d,C_0,\alpha_0,C_1} R^{-d}\|V^{\cdot k}(\cdot,y)\|_{L_1(B_{3R/8}^+(x^0))}\\
\label{161025@eq7a}
&\lesssim_{d,C_0,\alpha_0,C_1} R^{-d}\|V^{\cdot k}(\cdot,y)\|_{L_1(\bR^d_+\cap B_{R/2}(x))}.
\end{align}
Combining \eqref{161025@eq7} and \eqref{161025@eq7a}, we obtain \begin{equation}		\label{161025@eq3}
|V(x,y)|\lesssim_{d,C_0,\alpha_0,C_1} R^{-d}\|V(\cdot,y)\|_{L_1(\bR^d_+\cap B_{R/2}(x))}.
\end{equation}
\item
We now prove the estimate \eqref{161019@eq5}.
Let  $x,\,y\in \bR^d_+$ and $0<R:=|x-y|\le \min\{R_0,R_1\}$.
If $(u, p)\in \rY^1_2(\bR^d_+)^d\times L_2(\bR^d_+)$ satisfies
\[\left\{
\begin{aligned}
\cL^* u + \nabla p=f &\quad \text{in }\, \bR^d_+,\\
\divg u=0 &\quad \text{in }\, \bR^d_+,
\end{aligned}
\right.\]
where $f\in L_\infty(\bR^d_+)^d$ with $\operatorname{supp}f\subset (\bR^d_+\cap B_{R/2}(x))$,
then by the condition $(c)$ in Definition \ref{D23}, we have  
\begin{equation}		\label{161025@eq9}
u(y)=\int_{\bR^d_+\cap B_{R/2}(x)}V(z,y)^{\operatorname{tr}}f(z)\,dz.
\end{equation}
Moreover, since  
\[\left\{
\begin{aligned}
\cL^* u + \nabla p=0 &\quad \text{in }\, \bR^d_+\cap B_{R/2}(y),\\
\divg u=0 &\quad \text{in }\, \bR^d_+\cap B_{R/2}(y),\\
u=0 &\quad \text{on }\, \partial \bR^d_+,
\end{aligned}
\right.\]
we obtain that (see \eqref{161025@eq3})
\[\|u\|_{L_\infty(B_{R/16}(y))} \lesssim R^{-d}\|u\|_{L_1(\bR^d_+\cap B_{R/2}(y))}.\]
From this together with \eqref{160721@eq2}, we get 
\begin{align*}
\|u\|_{L_\infty(B_{R/16}(y))} 
&\lesssim_{d,\lambda,C_0,\alpha_0,C_1} R^{1-d/2}\|u\|_{L_{2d/(d-2)}(\bR^d_+\cap B_{R/2}(y))} \\
&\lesssim_{d,\lambda,C_0,\alpha_0,C_1} R^2\|f\|_{L_\infty(\bR^d_+\cap B_{R/2}(x))}.
\end{align*}
Combining this and \eqref{161025@eq9}, and then using the duality argument, we obtain 
\[\|V(\cdot,y)\|_{L_1(\bR^d_+\cap B_{R/2}(x))} \lesssim R^2,\]
which together with \eqref{161025@eq3} implies the desired estimate \eqref{161019@eq5}.
\item

To show estimates $i)$ -- $v)$ in Theorem \ref{MRE}, due to Corollary \ref{MRC},
we may consider only the case $y\in \bR^d_+$ and  
\[16d_y\le {R}\le \min\{R_0,R_1\}.\]
Take $y^0\in \partial \bR^d_+$ satisfying $\operatorname{dist}(y,\partial \bR^d_+)=|y-y^0|$.
Then 
\[\big(\bR^d_+\cap B_{R/16}(y)\big)\subset  B_{R/8}^+(y^0)\subset B_{R/2}^+(y^0)\subset (\bR^d_+\setminus B_R(y)).\]
Let $\eta$ be a smooth functions on $\bR^d$ satisfying
\[0\le \eta\le 1, \quad \eta\equiv 1 \,\text{ on }\, B_{R/4}(y^0), \quad \operatorname{supp}\eta\subset B_{R/2}(y^0), \quad |D\eta|\lesssim R^{-1}.\]
Like the estiamte \eqref{160807@eq1}, we have 
\begin{align*}
&\left|\int_{\bR^d_+}\Pi^k(\cdot,y)\divg \big((1-\eta^2)V^{\cdot k}(\cdot,y)\big)\,dx\right|\\
&\lesssim \int_{\cD^+}\big|\Pi^k(\cdot,y)-(\Pi^k(\cdot,y))_{\cD^+}\big|^2\,dx 
+ R^{-2}\int_{\cD^+}|DV^{\cdot k}(\cdot,y)|^2\,dx,
\end{align*}
where $\cD^+=B_{R/2}^+(y^0)\setminus \overline{B_{R/4}(y^0)}$.
Like the estimate \eqref{160809@eq1}, we have, by using Lemma \ref{160920@lem2} $(b)$,
\begin{equation}		\label{160929@eq1}
\int_{\bR^d_+}(1-\eta^2)|DV^{\cdot k}(\cdot,y)|^2\,dx 
\lesssim R^{-2} \int_{\cD_0^+}|V^{\cdot k}(\cdot,y)|^2\,dx,
\end{equation}
where $\cD^+_0=B^+_{5R/8}(y^0)\setminus \overline{B_{R/8}(y^0)}$.
Since
\[|x-y|\le \frac{5R}{8}, \quad \forall x\in \cD^+_0,\]
we apply \eqref{161019@eq5} to \eqref{160929@eq1} and then follow the same steps used in the proof of \eqref{160802@eq1}, we obtain the estimate $i)$.
The proof of $ii)$ and $iii)$ are the same as that of Lemma \ref{L53}.

We shall sketch the proof of $iv)$, which is similar to the proof of Lemma \ref{L54}.
Let $\varphi \in \rY^1_2(\bR^d_+)^d$ be a solution to the divergence equation
\[\divg \varphi=\Pi^k(\cdot,y)I_{\bR^d_+\setminus \overline{B_{R/2}(y^0)}} \quad \text{in }\, \bR^d_+\]
satisfying 
\[\|\varphi\|_{Y^1_2(\bR^d_+)} \lesssim_{d} \|\Pi^k(\cdot,y)\|_{L_2(\bR^d_+\setminus \overline{B_{R/2}(y^0)})}.\]
Using $(1-\eta)\varphi$ as a test function, we obtain 
\begin{equation}
\label{161011@eq1a}
\begin{split}
\int_{\bR^d_+\setminus \overline{B_{R/2}(y^0)}}|\Pi^k(\cdot,y)|^2\,dx
&\lesssim \int_{\cD^+}\big|\Pi^k(\cdot,y)-(\Pi^k(\cdot,y))_{\cD^+}\big|^2\,dx \\
&\quad + \int_{\bR^d_+\setminus \overline{B_{R/4}(y^0)}}|DV^k(\cdot,y)|^2\,dx.
\end{split}
\end{equation}
Since 
\[\cL V_\varepsilon^{\cdot k}(\cdot,y)+\nabla \Pi^k_\varepsilon(\cdot,y)=0 \quad \text{in }\, \cD^+,\]
it follows from Lemma \ref{160808@lem2} that 
\begin{equation}		\label{161011@eq2}
\int_{\cD^+}\big|\Pi^k(\cdot,y)-(\Pi^k(\cdot,y))_{\cD^+}\big|^2\,dx 
\lesssim \int_{\cD^+}|DV^{\cdot k}(\cdot,y)|^2\,dx.
\end{equation}
Note that $\cD^+\subset (\bR^d_+\setminus \overline{B_{R/8}(y)})$.
Combining \eqref{161011@eq1a} and \eqref{161011@eq2} we obtain 
\[\int_{\bR^d_+\setminus \overline{B_{R/2}(y^0)}}|\Pi^k_\varepsilon(\cdot,y)|^2\,dx 
\lesssim \int_{\bR^d_+\setminus \overline{B_{R/8}(y)}}|DV_\varepsilon^{\cdot k}(\cdot,y)|^2\,dx.\]
Thus, the desired estimate $iv)$ follows from $i)$.
We omit the proof of $v)$ because it is very similar.
\end{enumerate}
This completes the proof of Theorem \ref{MRE}.

\section{Proof of Theorem \ref{MRF}}
\label{S9}	

\begin{lemma}		\label{161125@lem1}
Let $\cL_0$ be the operator  in \eqref{161125@eq1} and let $0<r<R$.
If $(u,p)\in W^1_2(B_R^+)^d\times L_2(B_R^+)$ satisfies
\[\left\{
\begin{aligned}
\cL_0 u + \nabla p=0 &\quad \text{in }\, B_R^+,\\
\divg u=0 &\quad \text{in }\, B_R^+,\\
u=0 &\quad \text{on }\,  B_R \cap \partial \bR^d_+,
\end{aligned}
\right.\]
then 
\begin{equation}		\label{161125@eq1a}
\|Du\|_{L_\infty(B_r^+)}\lesssim_{d,\lambda} (R-r)^{-d/2}\|Du\|_{L_2(B_R^+)}
\end{equation}
and 
\begin{equation}		\label{161125@eq1b}
\|p\|_{L_\infty(B_r^+)}\lesssim_{d,\lambda} (R-r)^{-d/2}\big(\|Du\|_{L_2(B_2^+)}+\|p\|_{L_2(B_2^+)}\big).
\end{equation}
\end{lemma}

\begin{proof}
Using \cite[Lemma 4.4]{arXiv:1604.02690v2} and repeating the same arguments in the proofs of Lemma \ref{161014@lem5} and Corollary \ref{161014@cor5}, one can easily show that the estimates \eqref{161125@eq1a} and \eqref{161125@eq1b} hold.
We omit the details.
\end{proof}

We note that the following lemma is well known (see, for instance,  \cite{MR641818}). 
We present that for the sake of completeness.

\begin{lemma}[Reverse H\"older inequality]		\label{161127@lem1}
Let $\Omega_R(x^0)=\bR^d_+\cap B_R(x^0)$ with $x^0\in \bR^d_+$ and $R>0$.
If $(u, p)\in W^1_2(\Omega_R(x^0))^d\times L_2(\Omega_R(x^0))$ satisfies
\[\left\{
\begin{aligned}
\cL u + \nabla p=0 &\quad \text{in }\, \Omega_R(x^0),\\
\divg u=0 &\quad \text{in }\, \Omega_R(x^0),\\
u=0 &\quad \text{on }\,  B_R(x^0) \cap \partial \bR^d_+,
\end{aligned}
\right.\]
then there exists a constant $q_0=q_0(d,\lambda)>2$ such that 
\begin{equation}		\label{161127@eq1}
\left(\fint_{\Omega_{R/2}(x^0)}|Du|^{q_0}\,dx\right)^{1/q_0}
\lesssim_{d,\lambda} \left(\fint_{\Omega_R(x^0)}|Du|^2\,dx\right)^{1/2}.
\end{equation}
\end{lemma}

\begin{proof}
Throughout the proof, we regard $u$ as a function in $W^1_2(B_R)$ by setting $u=0$ in $B_R\setminus B_R^+$.
Set $q_1=2d/(d+2)$ and $U=|Du|^{q_1}$.
We claim that for any $y\in B_R(x^0)$, $0<r\le \operatorname{dist}(y,\partial B_R(x^0))$, and $0<\delta<1$, we have
\begin{equation}		\label{170313@eq1}
\fint_{B_{r/12}(y)}U^{2/q_1}\,dx 
\le \delta\fint_{B_r(y)}U^{2/q_1}\,dx 
+ C(d,\lambda,\delta) \left(\fint_{B_r(y)}U\,dx\right)^{2/q_1}.
\end{equation}
Let $y\in B_R(x^0)$ and $0<r\le \operatorname{dist}(y,\partial B_R(x^0))$.
We consider two cases when  $r/6\le \operatorname{dist}(y,\partial \bR^d_+)$ and $r/6>\operatorname{dist}(y,\partial \bR^d_+)$.
Assume that  $r/6\le \operatorname{dist}(y,\partial \bR^d_+)$.
Since it holds that  
\[\left\{
\begin{aligned}
\cL \big(u-(u)_{B_{r/6}(y)}\big) + \nabla p=0 &\quad \text{in }\, B_{r/6}(y),\\
\divg \big(u-(u)_{B_{r/6}(y)}\big)=0 &\quad \text{in }\, B_{r/6}(y),
\end{aligned}
\right.\]
by Lemma \ref{160920@lem2}, H\"older's inequality, and Poincar\'e's inequality, we have 
\begin{align*}
\|Du\|_{L_2(B_{r/12}(y))} 
&\lesssim r^{-1}\big\|u-(u)_{B_{r/6}(y)}\big\|_{L_2(B_{r/6}(y))}\\
&\lesssim r^{-1}\big\|u-(u)_{B_{r/6}(y)}\big\|_{L_{2d/(d-2)}(B_{r/6}(y))}\big\|u-(u)_{B_{r/6}(y)}\big\|_{L_{q_1}(B_{r/6}(y))}\\
&\lesssim \|Du\|_{L_2(B_r(y))}\|Du\|_{L_{q_1}(B_r(y))}.
\end{align*}
Using this together with Young's inequality, we obtain the estimate \eqref{170313@eq1}.
If $r/6>\operatorname{dist}(y,\partial \bR^d_+)$, then we take $y^0\in \partial \bR^d_+\cap B_{r/6}(y)$ satisfying $\operatorname{dist}(y^0,\partial \bR^d_+)=|y-y^0|$.
Since 
\[
B_{r/6}(y)\subset B_{r/3}(y^0)\subset B_{2r/3}(y^0)\subset B_r(y),
\]
Then by Lemma \ref{160920@lem2}, H\"older's inequality, and Poincar\'e's inequality (see, for instance, \cite[Eq. (7.45), p. 164]{MR1814364}), we have 
\begin{align*}
\|Du\|_{L_2(B_{r/6}(y))} &\le \|Du\|_{L_2(B_{r/3}(y^0))}\lesssim r^{-1}\|u\|_{L_2(B_{2r/3}(y^0))}\\
&\lesssim r^{-1}\|u\|_{L_{2d/(d-2)}(B_{2r/3}(y^0))}\|u\|_{L_{q_1}(B_{2r/3}(y^0))}\\
&\lesssim \|Du\|_{L_2(B_r(y))}\|Du\|_{L_{q_1}(B_r(y))}.
\end{align*}
Using this together with Young's inequality, we obtain the estimate \eqref{170313@eq1}.

We are now ready to prove the lemma.
By \eqref{170313@eq1} and a standard covering argument, we see that 
\[\fint_{B_{r/2}(y)}U^{2/q_1}\,dx 
\le \frac{1}{2}\fint_{B_r(y)}U^{2/q_1}\,dx 
+ C(d,\lambda) \left(\fint_{B_r(y)}U\,dx\right)^{2/q_1}\]
for any $B_r(y)\subset B_R(x^0)$.
Therefore, applying a version of Gehring's lemma (see, for instance, \cite[Lemma 4.5]{arXiv:1503.07290v3}) and using the definition of $U$,  we obtain that there exists $q_0>2$ satisfying \eqref{161127@eq1}.
This completes the proof.
\end{proof}

We only prove the case (b) of Theorem \ref{MRF} because (a) is its special case.
We recall the notation \eqref{161125@dd1}.
Assume that $\omega(R_1)\le \gamma$,
where $\gamma\in (0,1)$ is a constant to be chosen later.
Let $(u, p)\in W^1_2(B_R^+(x^0))^d\times L_2(B_R^+(x^0))$ satisfy \eqref{161025@eq1} with $x^0\in \partial \bR^d_+$ and $R\in (0, R_1]$.
Denote $\Gamma_R(x^0)=B_R(x^0)\cap \bR^d_+$.

Let $y=(0,y')\in \Gamma_R(x^0)$, $r>0$, and $B_{r}^+(y)\subseteq B_R^+(x^0)$.
Then by using Lemmas \ref{161125@lem1} and \ref{161127@lem1}, and following the same argument used in deriving \eqref{161015@eq2a}, we have 
\begin{equation}		\label{161125@dd2}
\|Du\|_{L_2(B_\rho^+(y))}\lesssim \left(\left(\frac{\rho}{r}\right)^{d/2}+\gamma\right)\|Du\|_{L_2(B_{r}^+(y))}
\end{equation}
for any $0<\rho\le 2r$.
Similar to \eqref{161015@eq2a}, we also have 
\begin{equation}		\label{161128@eq5}
\|Du\|_{L_2(B_\rho(y))}\lesssim \left(\left(\frac{\rho}{r}\right)^{d/2}+\gamma\right)\|Du\|_{L_2(B_{r}(y))}
\end{equation}
for any $B_{r}(y)\subset B_R^+(x^0)$ and $0<\rho\le r$.

Now we extend $u$ to $B_R(x^0)$ by setting $u\equiv 0$ on $B_R(x^0)\setminus B_R^+(x^0)$.
Then by \eqref{161125@dd2} and \eqref{161128@eq5}, one can easily obtain that 
$$
\|Du\|_{L_2(B_\rho(y))}\lesssim \left(\left(\frac{\rho}{r}\right)^{d/2}+\gamma\right)\|Du\|_{L_2(B_{r}(y))}
$$
for any $B_r(y)\subset B_R(x^0)$ and $0<\rho< r$.
Exactly the same steps as in the proof of Theorem \ref{MRD} yield the estimate \eqref{161019@aa1}.
This completes the proof of Theorem \ref{MRF}.

\section{Proof of Theorem \ref{MRG}}
\label{S10}	

We mainly follow the proof in Kang--Kim \cite[Theorem 3.13]{MR2718661}.
For $x\in \bR^d_+$ and $R\le R_2$, we denote $\Omega_R(x)= \bR^d_+\cap B_R(x)$.

\begin{enumerate}[\bf{Step} 1)]
\item
Assume that  $(u,p)\in W^1_2(\Omega_R(x^0))^d\times L_2(\Omega_R(x^0))$ satisfies \eqref{170311@eq1}, where $x^0\in \bR^d_+$ and $0<R\le R_2$ satisfying $d_{x^0}<R/4$.
Using Assumption \ref{ASSC} and the Poincar\'e inequality, we have 
\begin{align*}
\big[u\chi_{\Omega_{R}(x^0)}\big]_{C^{\alpha_2}(B_{R/2}(x^0))}&\lesssim R^{-\alpha_2}\left(\fint_{\Omega_{R}(x^0)}|u|^2\,dx \right)^{1/2}\\
&\lesssim R^{1-d/2-\alpha_2}\left(\int_{\Omega_R(x^0)}|Du|^2\,dx\right)^{1/2}.
\end{align*}
Let $z^0\in \partial B_{2d_{x^0}}(x^0)\setminus \bR^d_+$ and observe that $|z^0-x^0|<R/2$.
From the above inequality and the fact that 
\[
|u(x^0)|=\Big|u(x^0)-u\chi_{\Omega_{R}(x^0)}(z^0)\Big|\lesssim \big[u\chi_{\Omega_{R}(x^0)}\big]_{C^{\alpha_2}(B_{R/2}(x^0))}(d_{x^0})^{\alpha_2},
\]
we have 
\begin{equation}		\label{170309@eq1a}
|u(x^0)|\lesssim(d_{x^0})^{\alpha_2}R^{1-d/2-\alpha_2}\|Du\|_{L_2(\Omega_R(x^0))}
\end{equation}
for $x^0\in \bR^d_+$ and $0<R\le R_2$ satisfying $d_{x^0}<R/4$.

\item
In this step, we first claim that 
\begin{equation}		\label{170310@eq2}
|V(x,y)|\lesssim \min\{d_x,|x-y|\}^{\alpha_2}|x-y|^{2-d-\alpha_2}
\end{equation}
for any  $x,\, y\in \bR^d_+$ satisfying $0<|x-y|<R_2$.
Due to \eqref{161019@eq5}, it suffices to show that 
\begin{equation}		\label{170311@eq3}
|V(x,y)|\lesssim (d_x)^{\alpha_2}|x-y|^{2-d-\alpha_2} \quad \text{if }\, 4d_x<R:=\frac{|x-y|}{2}.
\end{equation}
By \eqref{170309@eq1a}, we have 
\[
|V(x,y)|\lesssim (d_x)^{\alpha_2}R^{1-d/2-\alpha_2}\|DV(\cdot,y)\|_{L_2(\Omega_{R}(x))}.
\]
Using this together with the estimate $iii)$ in Theorem \ref{MRE},
we have 
$$
|V(x,y)|\lesssim (d_x)^{\alpha_2}R^{1-d/2-\alpha_2}\|DV(\cdot,y)\|_{L_2(\bR^d_+\setminus B_R(y))}\lesssim (d_x)^{\alpha_2}R^{2-d-\alpha_2},
$$
which gives the estimate \eqref{170311@eq3}.

Next, we claim that 
\begin{equation}		\label{170310@eq5}
|V(x,y)|\lesssim \min\{d_x, |x-y|\}^{\alpha_2}\min\{d_y,|x-y|\}^{\alpha_2}|x-y|^{2-d-2\alpha_2}
\end{equation}
for any  $x,\, y\in \bR^d_+$ satisfying $0<|x-y|<R_2/2$.
We may assume that $4d_y<R:=|x-y|/4$ to prove \eqref{170310@eq5} because otherwise would follow from \eqref{170310@eq2}.
Using  Corollary \ref{161122@cor1}, \eqref{170309@eq1a}, and Caccioppoli's inequality (see, for instance, Lemma \ref{160920@lem2} $(b)$), we have 
\begin{align}	
\nonumber
|V(x,y)|&\lesssim (d_y)^{\alpha_2}R^{1-d/2-\alpha_2}\|D{}^*V(\cdot,x )\|_{L_2(\Omega_R(y))}\\
\label{170310@eq3}
&\lesssim (d_y)^{\alpha_2}R^{-d/2-\alpha_2}\|{}^*V(\cdot,x)\|_{L_2(\Omega_{2R}(y))}.
\end{align}
Since it holds that 
\[
2R<|x-z|<6R \quad \text{for all }\, z\in \Omega_{2R}(y),
\]
we obtain by \eqref{170310@eq2} and \eqref{170310@eq3} that 
\[
|V(x,y)|\lesssim (d_y)^{\alpha_2}\min\{d_x,|x-y|\}^{\alpha_2}R^{2-d-2\alpha_2}.
\]
\item
To prove the estimate \eqref{170304@eq1},  it suffices to show that 
\begin{equation}		\label{170310@eq9}
|V(x,y)|\lesssim \min\{d_x, R_2\}^{\alpha_2}\min\{d_y,R_2\}^{\alpha_2}R_2^{2-d-2\alpha_2}
\end{equation}
for any $x,\,y\in \bR^d_+$ satisfying $|x-y|\ge R_2/2$.
Set $R=R_2/4$.
Note that $(V(\cdot,y),\Pi(\cdot,y))$ satisfies 
$$
\left\{
\begin{aligned}
&\cL u+\nabla p =0, \quad \divg u=0 \quad \text{in }\, \Omega_{R}(x),\\
&u=0 \quad \text{on }\, \partial \bR^d_+.
\end{aligned}
\right.
$$
From Lemma \ref{161121@lem7} and $i)$ in Theorem \ref{MRE}, it follows that 
\begin{align*}
|V(x,y)|&\lesssim R^{-d}\|V(\cdot,y)\|_{L_1(\Omega_{R}(x))}\\
&\lesssim R^{(2-d)/2}\|V(\cdot,y)\|_{L^{2d/(d-2)}(\bR^d_+\setminus B_{R}(y))}\lesssim R_2^{2-d}.
\end{align*}
By utilizing the above inequality, and following  the same steps used in deriving \eqref{170310@eq5}, we concluded the estimate \eqref{170310@eq9}.
\end{enumerate}
This completes the proof of Theorem \ref{MRG}.

\section{Green functions on unbounded domains}
\label{S11}	
In this section we consider the existence of the Green function for the Stokes system on a domain $\Omega$ with $|\Omega|=\infty$. 
We impose the following assumption on $\Omega$ in Theorem \ref{MRH} below.

\begin{assumption}		\label{ASSD}
There exists a constant $C_3>0$ such that the following holds:
for any $g\in L_2(\Omega)$, there exists $u\in \mathring{Y}^1_2(\Omega)^d$ satisfying
\[
\divg u=g \quad \text{in }\, \Omega, \quad \|Du\|_{L_2(\Omega)}\le C_3 \|g\|_{L_2(\Omega)}.
\]
\end{assumption}

\begin{remark}		
Below are some examples of cases when Assumption \ref{ASSD} holds.
\begin{enumerate}[$(i)$]
\item
$\Omega$ is the whole space or half space.
More generally, 
\[
\Omega=\{x\in \bR^d:x_1>0, \, x_2>0,\, \text{ or }\, x_d>0\}.
\]
\item
$\Omega$ is a locally Lipschitz and exterior domain (see \cite[Theorem III.3.6]{MR2808162}).
\end{enumerate}
\end{remark}

\begin{remark} 			\label{170314@rmk1}
Note that if  $\Omega$ is a domain in $\bR^d$, $d\ge 3$, with $|\Omega|=\infty$,
then under Assumption \ref{ASSD}, we obtain the $L_2$-solvability of the Stokes systems (with measurable coefficients) and the estimate \eqref{160721@eq2}. 
\end{remark}

Under Assumptions \ref{ASSA} and \ref{ASSD}, using Remark \ref{170314@rmk1} and  repeating the same arguments in the proof of Theorem \ref{MRB}, one can  prove the existence of the Green function on $\Omega$.
We think it is worth to present the precise statement.
We denote $d_x=\operatorname{dist}(x,\partial \Omega)$ for $x\in \Omega$.

\begin{theorem}		\label{MRH}
Let $\Omega$ be a domain in $\bR^d$, $d\ge 3$, with $|\Omega|=\infty$.
If Assumptions \ref{ASSA} and \ref{ASSD} hold, then there exists a unique Green function $((V(x,y),\Pi(x,y))$ for the Stokes operator on $\Omega$.
Moreover, for any $x,\, y\in \Omega$ satisfying $0<|x-y|\le \min\{d_x,d_y,R_0\}$, we have 
\[
|V(x,y)|\lesssim_{d,\lambda,C_0,\alpha_0,C_3}|x-y|^{2-d}.
\]
Furthermore, the Green function satisfies the representation formulas \eqref{170307@eq1} and \eqref{161122@eq8}, and it also satisfies the estimates $i)$ -- $v)$ in Corollary \ref{MRC}.
\end{theorem}

By modifying the proof of \eqref{161019@eq5}, one can prove the following pointwise bound.

\begin{theorem}		\label{MRI}
Let $\Omega$ be a domain in $\bR^d$, $d\ge 3$, with $|\Omega|=\infty$.
Suppose that Assumptions \ref{ASSA} and \ref{ASSD} hold.
Let $(V(x,y),\Pi(x,y))$ be the Green function constructed in Theorem \ref{MRH}.
If Assumption \ref{ASSB} holds with $\Omega$  in place of $\bR^d_+$, respectively, then for any $x,\, y\in \Omega$ satisfying $0<|x-y|\le \min\{R_0,R_1\}$, we have 
\[
|V(x,y)|\lesssim_{d,\lambda,C_0,\alpha_0,C_3,C_1}|x-y|^{2-d}.
\]
\end{theorem}

We note that Caccioppoli's inequality holds for the Stokes system on a Lipschitz domain.
Then by following the proof of  Theorem \ref{MRG}, we obtain the following estimate.

\begin{theorem}
Let $\Omega$ be a  domain in $\bR^d$, $d\ge 3$, with $|\Omega|=\infty$.
Suppose that $\Omega$ has a Lipschitz boundary with a bounded Lipschitz constant.
If  Assumption \ref{ASSD} holds, and if Assumption  \ref{ASSC} holds with $\Omega$ in place of $\bR^d_+$, then for any $x,\,y\in \Omega$ with $x\neq y$, 
$$
|V(x,y)|\le C\min\{d_x,|x-y|,R_2\}^{\alpha_2}\min\{d_y,|x-y|,R_2\}^{\alpha_2} \min\{|x-y|,R_2\}^{2-d-2\alpha_2},
$$
where $C=C(d,\lambda,C_2,\alpha_2, C_3,\Omega)$.
\end{theorem}

\section*{Acknowledgment}
The authors would like to express their sincerely gratitude to the referee for careful reading and for many helpful comments and suggestions.
We also thank Tongkeun Chang for valuable comments.
J. Choi was supported by Basic Science Research Program through the National Research Foundation of Korea (NRF) funded by the Ministry of Education (2014R1A1A2054865).
M. Yang has been supported by the National Research Foundation of Korea (NRF) grant funded by the Korea government(MSIP) (No. 2016R1C1B2015731).


\end{document}